\documentclass[a4paper,11pt]{amsart}

\usepackage[utf8]{inputenc}
\usepackage[dvips]{graphicx}
\usepackage[shortlabels]{enumitem}
\usepackage{amsmath,amssymb,color,enumitem,graphics,hyperref,latexsym,pgfplots,tikz,vmargin}
\usepackage{relsize}
\usepackage{dsfont}
\usepackage{mathrsfs}
\usepackage[mathscr]{eucal}
\usepackage{bbm}

\hypersetup{
unicode=false,          
pdftoolbar=true,        
pdfmenubar=true,        
pdffitwindow=false,     
pdfstartview={FitH},    
pdftitle={Asymptotic and continuous group Orlicz cohomology},    
pdfauthor={Yaroslav Kopylov, Emiliano Sequeira},     
pdfsubject={},   
pdfkeywords={}, 
pdfnewwindow=true,      
colorlinks=true,       
linkcolor=black,          
citecolor=black,        
filecolor=magenta,      
urlcolor=cyan           
}

\def\Dd{{\mathcal{D}}}

\def\Hh{{\mathcal{H}}}

\def\N{{\mathbb{N}}}

\def\R{{\mathbb{R}}}

\def\Z{{\mathbb{Z}}}

\def\H{{\mathbb{H}}}

\newcommand{\Ker}{\operatorname{Ker}}
\newcommand{\I}{\operatorname{Im}}
\newcommand{\Loc}{\mathbb{L}^\phi_{loc}}

\newtheorem{lemma}{Lemma}[section]
\newtheorem{proposition}[lemma]{Proposition}
\newtheorem{theorem}[lemma]{Theorem}
\newtheorem{corollary}[lemma]{Corollary}

\theoremstyle{definition}

\theoremstyle{remark}
\newtheorem{remark}[lemma]{Remark}

\begin{document}

\title{On Asymptotic and Continuous group Orlicz Cohomology}
\author[Ya. Kopylov and E. Sequeira]{Yaroslav Kopylov \and Emiliano Sequeira}
\address[\textsc{Yaroslav Kopylov}]{Sobolev Institute of Mathematics, 4 Acad. Koptyug Ave., Novosibirsk 630090, Russia.}
\email{yakop@math.nsc.ru; \quad yarkopylov@gmail.com}
\address[\textsc{Emiliano Sequeira}]{Universidad de la República; Montevideo, Uruguay.}
\email{esequeira@cmat.edu.uy}


\begin{abstract}
We generalize some results on asymptotic and continuous group $L^p$-cohomology to Orlicz cohomology. In particular, we show that asymptotic Orlicz cohomology is a quasi-isometry invariant and that both notions coincide in the case of a locally compact second countable group.  The case of degree $1$ is studied in more detail.    
\end{abstract}

\maketitle

\section{Introduction}

Different versions of $L^p$-cohomology (and, more generally, $L^{p,q}$-cohomology) have been studied in last decades with the aim of obtaining Lipschitz and  quasi-isometry invariants and explore the existence of inequalities of Sobolev-Poincaré type and $p$-harmonic functions. This notion is defined, for instance, for simplicial complexes \cite{BP, E, G93}, Riemannian manifolds \cite{A,GKSh82_1, GKSh82_2, Pa88}, discrete and topological groups \cite{BMV05, BR,BR2, CT, G93, MV, Puls, T} and more general metric measure spaces \cite{G,Pa88,Pa95,T}, and consists, in all cases, of a family of topological vector spaces constructed from a cochain complex of $L^p$-integrable graded functions.

As for classical $L^p$-spaces, one can generalize $L^p$-cohomology by using Orlicz spaces, which are obtained from a convex function (more precisely, a Young function) $\phi$ instead of the parameter $p$. A motivation to do this is to obtain a bigger family of quasi-isometry invariants, which can be useful, for example, for distinguishing certain spaces up to quasi-isometry, as is done in \cite{C}. 

In particular, asymptotic $L^p$-cohomology is a construction introduced by Pansu in \cite{Pa95}, following a previous version for degree $1$ \cite{Pa88}, defined for a metric measure space with bounded geometry. A quite complete study of this notion can be read in \cite{G}. Asymptotic $L^p$-cohomology provides a quasi-isometry invariant for a wide family of metric spaces, however, it has the disadvantage of being difficult to compute. 

We study the Orlicz version of this notion and prove the following result, where $L^\phi H^k_{AS}(X)$ denotes the $k$-space of asymptotic Orlicz cohomology of a metric space $X$ for a Young function $\phi$, and $L^\phi \overline{H}^k_{AS}(X)$ is the respective reduced space. For a proof in the $L^p$-case see \cite{G,Pa95}. 

\begin{theorem}\label{MainInvariance}
Let $(X,\mu)$ and $(Y,\nu)$ be two metric measure spaces with bounded geometry and $\phi$ a Young function. If there exits a quasi-isometry $F:X\to Y$, then $L^\phi H^*_{AS}(X)$ and $L^\phi H^*_{AS}(Y)$ are isomorphic (as topological vector spaces) and $L^\phi \overline{H}^*_{AS}(X)$ and $L^\phi \overline{H}^*_{AS}(Y)$ are isomorphic (as Fr\'echet spaces).
\end{theorem}

Recent articles \cite{BR,BR2} by Bourdon and Rémy study the continuous group $L^p$-cohomology, following some previous ideas given in \cite{CT,E2,G93}, which is defined for locally compact groups. They prove an equivalence theorem between continuous group $L^p$-cohomology and asymptotic $L^p$-cohomology, which allows to conclude that the first one is a quasi-isometry invariant, and make some computations for Lie groups. 

We present an Orlicz version of the main result in \cite{BR}, which was earlier proved for the $L^2$ case in \cite{SS}. Here $H^k_{\mathrm{ct}}\bigl(G,L^\phi(G)\bigr)$ is the $k$-space of continuous group cohomology of $G$ with coefficients in $L^\phi(G)$, and $\overline{H}^k_{\mathrm{ct}}\bigl(G,L^\phi(G)\bigr)$ is the corresponding reduced cohomology space.

\begin{theorem}\label{MainEquivalence}
Suppose that $G$ is a~locally compact second countable group equipped with a~left-invariant proper metric and a left-invariant Haar measure and $\phi$ is a doubling Young function. Then the topological vector spaces $H^k_{\mathrm{ct}}\bigl(G,L^\phi(G)\bigr)$ and $L^\phi H^k_{AS}(G)$ are isomorphic for every $k\in\N$ and so are the Fr\'echet spaces $\overline{H}^k_{\mathrm{ct}}\bigl(G,L^\phi(G)\bigr)$ and $L^\phi \overline{H}^k_{AS}(G)$.
\end{theorem}

The doubling condition on $\phi$ is an assumption about its behaviour on $0$ and $\infty$ that will be specified later.  The existence of a left-invariant proper metric compatible with the topology of $G$ is guaranteed by Struble's theorem (see \cite[Theorem 2.B.4]{CH}).

Combining Theorems~\ref{MainInvariance} and~\ref{MainEquivalence}, we obtain the following result:

\begin{corollary}
If $F:G_1\to G_2$ is a quasi-isometry between two groups as in Theorem \ref{MainEquivalence} and $\phi$ is a doubling Young function, then $H^k\bigl(G_1,L^\phi(G_1)\bigr)$ is isomorphic to $H^k\bigl(G_2,L^\phi(G_2)\bigr)$ for every $k\in\N$. The same holds for the reduced cohomology.
\end{corollary}

For the case of degree $1$, we generalize some results given in \cite{MV,Puls,T} for $L^p$-cohomology and in \cite{KP2} for Orlicz cohomology in the case of discrete groups. In particular, we prove that if $G$ is compactly generated and $\phi$ satisfies some conditions, then every class in $H^1\bigl(G,L^\phi(G)\bigr)$ is represented by one (and only one) $\phi$-harmonic function.

Finally, we show with an example that some properties of Orlicz cohomology fail to hold if the Young function is not doubling. In particular, it is known that, if $\phi$ is doubling, then
\begin{itemize}
    \item the Orlicz cohomology in degree 1 of a uniformly contractible Gromov-hyperbolic   simplicial complex with bounded geometry whose boundary admits an Ahlfors-regular visual metric is reduced, that is, it coincides with its reduced Orlicz cohomology (see \cite{C});
    \item the continuous Orlicz cohomology in degree 1 of a non-amenable non-compact second countable locally compact group is reduced (see \cite{Ko13}). Therefore, its asymptotic Orlicz cohomology is reduced.
\end{itemize}
We prove that in both cases, the doubling condition is necessary.


\section*{Acknowledgments} 

The~work of~Ya.~Kopylov was carried out in the framework 
of the State Task to~the~Sobolev Institute of Mathematics (Project FWNF—2022—0006). E.~Sequeira was supported by  the \textit{Mathematical Center in Akademgorodok} under the agreement No. 075-15-2019-1675 with the \textit{Ministry of Science and Higher Education} of the Russian Federation.

We thank Pierre Pansu, Michael Puls, and Romain Tessera for~useful comments on~their works concerning the~$L^p$-case and Rafael Potrie for a valuable suggestion to improve the text.

This work is inspired by previous works by Pierre Pansu, Luc Genton, Marc Bourdon and Betrand Rémy.

\section{Preliminaries}

\subsection{Quasi-isometries}

Consider two metric spaces $X$ and $Y$, where the metric in both cases is denoted by $|\cdot-\cdot|$. A function $F:X\to Y$ is a \textit{quasi-isometry} if there exist two constants $\lambda \geq 1$ and $\epsilon\geq 0$ such that
\begin{enumerate}
\item[(a)] for every $x,x'\in X$, 
$$\lambda^{-1}|x-x'|-\epsilon\leq |F(x)-F(x')|\leq \lambda|x-x'|+\epsilon;$$
\item[(b)] for every $y\in Y$ there exists $x\in X$ such that $|F(x)-y|\leq\epsilon$.
\end{enumerate}
Notice that (a) is a coarse version of the bi-Lipschitz condition, while (b) expresses a kind of surjectivity.  

The notion of quasi-isometry defines an equivalence relation among metric spaces. Indeed, the composition of quasi-isometries is a quasi-isometry and for every quasi-isometry $F:X\to Y$ there exists a quasi-isometry $\overline{F}:Y\to X$ such that $F\circ \overline{F}$ and $\overline{F}\circ F$ are at bounded uniform distance from the identity. In this case, we say that $\overline{F}$ is a \textit{quasi-inverse} of $F$. Observe that the quasi-inverse is not uniquely defined, but one can easily show that two quasi-inverses of the same quasi-isometry are at bounded uniform distance from each other.

We refer to \cite{GH} for more details. 

\subsection{Orlicz spaces}

By a \textit{Young function} we mean a non-negative function $\phi:\R\to [0,+\infty)$ that is convex and even and satisfies $\phi(t)=0$ if and only if $t=0$. We say that $\phi$ is a $N$-function if it is in addition continuous and satisfies 
$$\lim_{t\to 0} \frac{\phi(t)}{t}=0\text{ and }\lim_{t\to+\infty} \frac{\phi(t)}{t}=+\infty.$$

If $(Z,\mu)$ is a measure space and $f:Z\to\R$ is a measurable function, we define
\begin{equation}\label{modular}
\rho_\phi(f)=\int_Z \phi\bigl(f(x)\bigr)\,d\mu(x).    
\end{equation}
The Orlicz space of $(Z,\mu)$ associated to $\phi$ is the space $L^\phi(Z)=L^\phi(Z,\mu)$ of classes of functions $f:Z\to\R$ such that $\rho_\phi(f/\alpha)<+\infty$ for some constant $\alpha>0$, equipped with the Luxembourg norm
$$\|f\|_\phi=\inf\left\{\alpha>0 : \rho_\phi\left(\frac{f}{\alpha}\right)\leq 1\right\}.$$
The space $\bigl(L^\phi(Z),\|\ \|_\phi\bigr)$ is a Banach space and, as in the $L^p$-case, the convexity of $\phi$ implies that $L^\phi(Z)\subset L^1_{loc}(Z)$. 

\begin{remark}\label{equivNormas}
Observe that if $\lambda\geq 1$ and $f$ is a measurable function on $Z$, then $\|f\|_\phi\leq \|f\|_{\lambda\phi}$. Moreover, the convexity of $\phi$ implies that $\phi(t/\lambda)\leq \phi(t)/\lambda$ for every $t\in\R$; thus, if $\alpha>\|f\|_\phi$, then
$$\int_Z \lambda\phi\left(\frac{f}{\lambda\alpha}\right) d\mu\leq \int_{Z}\phi\left(\frac{f}{\alpha}\right)d\mu\leq 1,$$
which implies $\|f\|_{\lambda\phi}\leq \lambda\|f\|_{\phi}$. We conclude that for any $\lambda>0$,
$$C^{-1}\|\ \|_\phi\leq \|\ \|_{\lambda\phi}\leq C\|\ \|_\phi,$$
where $C=\max\{\lambda,\lambda^{-1}\}$.

A consequence of this fact is that, if $\rho_\phi(f)\leq \lambda\rho_\phi(\lambda' g)=\rho_{\lambda\phi}(\lambda'g)$ for $\lambda,\lambda'>0$, then $\|f\|_\phi\leq C\lambda'\|g\|_\phi$.
\end{remark}






If $\phi$ is a Young function one can consider its convex conjugate
$$\psi:\R\to[0,+\infty],\ \psi(s)=\sup\{t|s|-\phi(t):t\geq 0\}.$$
It is easy to see that if $\phi$ is an $N$-function, then $\psi$ is also an $N$-function. A general version of Hölder's inequality holds for a pair of conjugate $N$-functions $(\phi,\psi)$: 
\begin{equation}\label{Holder}
\int_Z |fg|\,d\mu\leq 2\|f\|_{L^\phi}\|g\|_{L^\psi}    \end{equation}
for every $f\in L^\phi(Z)$ and $g\in L^\psi(Z)$. It is obtained by using Young's inequality:
\begin{equation}\label{YoungInequality}
ts\leq \psi(t)+\phi(s)\ \forall t,s\in\R.    
\end{equation}
We refer to \cite[Section 3.3]{RR} for a proof of \eqref{Holder}.

A Young function $\phi$ is \textit{doubling} if there exists a constant $D\geq 2$ such that for every
$t\geq 0$,
\begin{equation}\label{ConstDoblante}
\phi(2t)\leq D\phi(t).    
\end{equation}
(Observe that, since $\phi$ is convex and $\phi(0)=0$, then $\phi(2t)\geq 2\phi(t)$ for every $t\geq 0$.) It is not difficult to prove that $\phi$ is doubling if and only if there exists an increasing function $D_1:[1,+\infty)\to [1,\infty)$ such that for every $t\geq 0$ and $s\geq 1$,
$$\phi(st)\leq D_1(s)\phi(t).$$


The following proposition is known and easy to prove. A short proof can be found in \cite[Lemma 2.5.4]{Seq}. 

\begin{proposition}\label{PropDoblante}
Let $\phi$ be a doubling Young function, then
\begin{enumerate}
\item[(i)] $f\in L^\phi(Z,\mu)$ if and only if $\rho_\phi(f)<+\infty$.
\item[(ii)] $f_n\to f$ in $L^\phi(Z,\mu)$ if and only if $\rho_\phi(f_n-f) \to 0$. 
\end{enumerate}
\end{proposition}

In the same way as for $L^p$-spaces one can prove that simple functions are dense in $L^\phi(Z)$ if $\phi$ is doubling. For that it is necessary to use part $(ii)$ of Proposition \ref{PropDoblante}. This allows to prove the following fact by reproducing the corresponding proof for $L^p$-spaces.


\begin{lemma}\label{LemaCompSupp}
Suppose that $\phi$ is a doubling Young function. If $X$ is a proper metric space (i.e. every closed bounded set is compact) and $\mu$ is a Radon measure, then the space of continuous functions with compact support on $X$ is dense in $L^\phi(X,\mu)$. 
\end{lemma}

For more details on Orlicz spaces we refer to \cite{RR}. 

\subsection{Continuous group cohomology}

Let $G$ be a locally compact second countable group. A \textit{topological $G$-module} (or simply \textit{G-module}) is a pair $(\pi,V)$, where $V$ is a Hausdorff locally convex topological vector space over $\R$ and $\pi$ is a continuous representation of $G$ on $V$ (that is, $G\times V\to V, (g,v)\mapsto \pi(g)v$ is continuous).

For $k\in\N$ we consider the space $$C(G^{k+1},V)=\{\omega:G^{k+1}\to V : \omega\text{ is continuous}\}$$ 
equipped with the compact-open topology. 


The sum and product on $C(G^{k+1},V)$ are continuous. Furthermore, $C(G^{k+1},V)$ is Hausdorff and locally convex and the  representation $\Pi:G\to \mathrm{Aut}\bigl(C(G^{k+1},V)\bigr)$ defined by 
$$\bigl(\Pi(g)\omega\bigr)(x_0,\ldots,x_k)=(g\cdot \omega)(x_0,\ldots,x_k)=\pi(g)\bigl(\omega(g^{-1}x_0,\ldots,g^{-1}x_k)\bigr)$$
is continuous. We say that $\omega\in C(G^{k+1},V)$ is \textit{$G$-invariant} if $(g\cdot \omega)=\omega$ for every $g\in G$ and denote by $C(G^{k+1},V)^G$ the space of $G$-invariant functions.

In general, if $X$ is any set, $Y$ is a vector space and $A\subset X^{k+1}$, one can consider the \textit{(formal) derivative} of any function $f:A\to Y$,
\begin{equation}\label{DerivadaFormal}
d_kf(x_0,\ldots,x_{k+1})=\sum_{i=0}^{k+1} (-1)^i f(x_0,\ldots,\hat{x}_i,\ldots,x_{k+1}),    
\end{equation}
defined for $(x_0,\ldots,x_{k+1})$ in some subset of $X^{k+2}$. We also write $d=d_k$ when the sub-index is clear. 

Let us focus on the derivative of elements of $C(G^{k+1},V)$ for $k\geq 0$. It is easy to see that $d_k:C(G^{k+1},V)\to C(G^{k+2},V)$ is well-defined and continuous and maps $C(G^{k+1},V)^G$ onto $C(G^{k+2},V)^G$. Then we consider the complex
\begin{equation}\label{Comp.}
C(G^1,V)^G\stackrel{d_0}{\rightarrow}C(G^2,V)^G\stackrel{d_1}{\rightarrow}C(G^3,V)^G\stackrel{d_2}{\rightarrow}\cdots.  
\end{equation}
Its cohomology is called the \textit{continuous (group) cohomology of $G$ with coefficients in $(\pi,V)$}. That is, the family of topological vector spaces
$$H^k(G,V)=\frac{\mathrm{Ker}\ d_k}{\mathrm{Im}\ d_{k-1}}.$$
The \textit{reduced continuous (group) cohomology} of $G$ is the family of $G$-modules
$$\overline{H}^k(G,V)=\frac{\mathrm{Ker}\ d_k}{\overline{\mathrm{Im}\ d_{k-1}}}.$$

By a $G$\textit{-morphism} we mean a linear continuous map between $G$-modules that is $G$-equivariant, that is, $\varphi:A\to B$ such that $\varphi(g\cdot a)=g\cdot \varphi(a)$ for every $g\in G$ and $a\in A$.
A $G$-morphism is a \textit{strong $G$-injection} if it has a continuous left inverse. We say that a $G$-morphism  $\varphi:A\to B$ is \textit{strong} if the induced maps $\mathrm{Ker}\ \varphi\to A$ and $A/\mathrm{Ker}\ \varphi \to B$ are strong $G$-injections. From this we can define a \textit{strong resolution} of a $G$-module $V$ as an exact sequence of $G$-modules and strong $G$-morphisms
\begin{equation}\label{Resolution}
    0\to V \stackrel{d_{-1}}{\rightarrow} A^0\stackrel{d_0}{\rightarrow} A^1\stackrel{d_1}{\rightarrow}\cdots
    \end{equation}
We use the notation $0\to V\stackrel{d_*}{\rightarrow} A^*$ to mean a resolution as above.

\begin{remark}\label{ResolucionFuerte}
To see that \eqref{Resolution} is a strong resolution of $V$ it is enough to show that it admits a continuous contracting homotopy. That is, a family of linear continuous maps $\{h_k\}_{k\geq 0}$ such that
$$\left\{\begin{array}{cc} 
h_0\circ d_{-1} = \mathrm{Id} &  \\
h_{k+1}\circ d_k + d_{k-1}\circ h_k = \mathrm{Id} & \text{ for }k\geq 0.
\end{array}
\right.$$
Indeed, $d_{k-1}\circ h_k$ is the left inverse of $\mathrm{Ker}\ d_k\to A^k$ for $k\geq 0$, and $h_{k+1}$ induces the left inverse of $A^k/\mathrm{Ker}\ d_k\to A^{k+1}$ for $k\geq -1$, by putting $A^{-1}=V$.
\end{remark}

A $G$-module $U$ is \textit{relatively injective} if for every strong $G$-injection $\iota:A\to B$ and $G$-morphism $\varphi:A\to U$, there exists a $G$-morphism $\bar{\varphi}:B\to U$ such that $\bar{\varphi}\circ\iota=\varphi$. A strong $G$-resolution as \eqref{Resolution} is \textit{relatively injective} if $A^k$ is relatively injective for every $k\in\N$.

    
   
    
    
    


An important example of relatively injective strong resolution is given by \eqref{Comp.}. 

\begin{proposition}\label{PropSR}
The complex $0\to V\stackrel{d_*}{\rightarrow}C(G^{*+1},V)$ 
is a relatively injective strong $G$-resolution of $V$, where $d_{-1}(v)\equiv v$.
\end{proposition}

\begin{proof}
See \cite[Example 2.2]{BR}.
\end{proof}

The main technique to prove Theorem \ref{MainEquivalence} will be to find a special relatively injective strong $G$-resolution of $L^\phi(G)$, as it is done for the $L^p$-case by Bourdon and Rémy in \cite{BR}, and use the following results.

\begin{proposition}\label{PropEq}
Let $V$ a topological $G$-module. Assume that $0\to V\to A^*$ and $0\to V\to B^*$ are two relatively injective strong $G$-resolutions of $V$. Then the complexes $(A^*)^G$ and $(B^*)^G$ are homotopy equivalent. 
    
\end{proposition}

The proof of Proposition \ref{PropEq} can be found in \cite[p. 177, Proposition 1.1]{Gui}.

Combining Propositions \ref{PropSR} and \ref{PropEq} we obtain: 

\begin{corollary}\label{corRes}
Suppose that $0\to V\to A^*$ is a relatively injective strong $G$-resolution of $V$. Then the cohomology and the reduced cohomology of the complex $(A^*)^G$ are topologically isomorphic to $H^*(G,V)$ and $\overline{H}^*(G,V)$ respectively.
\end{corollary}

\section{Asymptotic Orlicz cohomology}

Let $(X,|\cdot-\cdot|)$ be a metric space equipped with a Borel measure $\mu$ satisfying the \textit{bounded geometry} condition: there exist $r_0>0$ such that
\begin{equation}\label{GeomAcotada}
0<v(r)=\inf\{\mu(B(x,r)):x\in X\}\leq V(r)=\sup\{\mu(B(x,r)):x\in X\}<+\infty 
\end{equation}
for every $r\geq r_0$, where $B(x,r)$ is the open ball of center $x$ and radius $r>0$.

We regard the product space $X^{k+1}$ as a set of $k$-simplices, so it is natural to consider the vector space of $k$-chains
$$C^k(X)=\left\{\sum_{i=1}^m a_i\Delta_i : m\in\N\text{ and }\Delta_i\in X^{k+1}, a_i\in\R\ \forall i=1,\ldots,m\right\},$$
and the boundary operator $\partial: C^k(X)\to C^{k-1}(X)$, defined on $X^{k+1}$ by
\begin{equation}\label{borde}
\partial\Delta=\sum_{i=1}^k \partial_i \Delta,    
\end{equation}
where $\partial_i\Delta =(x_0,\ldots,\hat{x}_i,\ldots,x_k)$ if $\Delta =(x_0,\ldots,x_k)$.

Given $k\in\N$, we equip $X^{k+1}$ with the product measure $\mu^{k+1}=\mu\times\cdots\times\mu$ and the distance
$$|\Delta-\Delta'|=\max\{|x_i-x'_i|:i=0,\ldots,k\}$$
for $\Delta=(x_0,\ldots,x_k)$ and $\Delta=(x'_0,\ldots,x'_k)$. Observe that $\mu^{k+1}$ satisfies 
$$v(r)^{k+1}\leq\mu^{k+1}\bigl(B(\Delta,r)\bigr)\leq V(r)^{k+1}$$
for every $\Delta\in X^{k+1}$ and $r>0$. In order to simplify the notation, we write $d\mu(x)=dx$ and $d\mu^{k+1}(\Delta)=d\Delta$.

For $s>0$ we define
$$X_s^{k+1}=\bigl\{\Delta\in X^{k+1}: \mathrm{diam}(\Delta)\leq s\bigr\}\subset X^{k+1}.$$

Fix a Young function $\phi$, then for every Borel function $u:X^{k+1}\to\R$ and $s>0$ we consider the semi-norm
\begin{equation}\label{semi-norms}
\|u\|_{\phi,s}=\inf\left\{\alpha>0 : \rho_{\phi,s}\left(\frac{u}{\alpha}\right)\leq 1 \right\}, \text{ where }
\rho_{\phi,s}\left(\frac{u}{\alpha}\right):=\int_{X_s^{k+1}} \phi\left(\frac{u(\Delta)}{\alpha}\right)d\Delta.    
\end{equation}
Then we define the $L^\phi$-space of \textit{Alexander-Spanier $k$-cochains} as the space $AS_\phi^k(X)$ of classes of measurable functions $u:X^{k+1}\to \R$ such that $\|u\|_{\phi,s}<+\infty$ for every $s>0$, equipped with the topology induced by the family of semi-norms $\{\|\ \|_{\phi,s}\}_{s>0}$. Observe that each semi-norm $\|\ \|_{\phi,s}$ is the $L^\phi$-norm in the space $X_s^{k+1}$.

An element $u\in AS_\phi^k(X)$ (or a function $u:X^{k+1}\to\R$ in general) can be linearly extended (a.e.) to $\tilde{u}:C^k(X)\to\R$. We will not distinguish between the function $u$ and its extension $\tilde{u}$ from now on.

\begin{remark}
For $t\leq s$ we consider the continuous operator 
$$T_{s,t}:L^\phi\bigl(X_s^{k+1}\bigr)\to L^\phi\bigl(X_t^{k+1}\bigr),\ u\mapsto u|_{X_t^{k+1}}$$
and take the inverse limit
$$\lim_{\leftarrow} L^\phi\bigl(X_s^{k+1}\bigr)=\left\{\{u_s\}\in \prod_{s>0}L^\phi\bigl(X_s^{k+1}\bigr): T_{s,t}(u_s)=u_t\text{ if }t<s\right\}$$
equipped with the topology induced by the family of semi-norms 
$\left\|\{u_s\}\right\|_s=\|u_s\|_\phi$ for $s>0$. It is a Fr\'echet space because it is defined from a dense projective system of Banach spaces (see \cite[Section 3.3.3]{Gol}). This is the Orlicz version of the definition given by Pansu in \cite{Pa95}.

Furthermore, the map
\begin{equation}\label{IsomAS}
AS_\phi^k(X)\to \lim_{\leftarrow} L^\phi\bigl(X_s^{k+1}\bigr),\ u\mapsto \bigl\{u|_{X_s^{k+1}}\bigr\}
\end{equation}
is clearly an isomorphism of topological vector spaces; hence $AS_\phi^k(X)$ is a Fréchet space for every $k$.
\end{remark}


Observe that the formal derivative on $AS_\phi^*(X)$ defined by \eqref{DerivadaFormal} satisfies $du(\Delta)=u(\partial\Delta)$. In this case it can also be called \textit{the coboundary operator} on $AS_\phi^*(X)$.

\begin{proposition}\label{Cont}
The derivative $d_k$ maps continuously $AS_\phi^k(X)$ to $AS_\phi^{k+1}(X)$. Moreover, $d_{k+1}\circ d_k=0$ for every $k\geq 0$ (which is also written as $d^2=0$).
\end{proposition}

\begin{proof}
The condition $d^2=0$ can be verified directly. Let us prove that if $u\in AS_\phi^k(X)$, then $\|du\|_{\phi,s}\preceq \|u\|_{\phi,s}$, which means that there exists a constant $C>0$ such that $\|du\|_{\phi,s}\leq C\|u\|_{\phi,s}$. In this case, the constant  depends on $s$ and $k$.

Using Jensen's inequality, we have
\begin{align*}
\int_{X_s^{k+2}}\phi\left(du\right)\,d\mu^{k+1} 
&\leq \frac{1}{k+2}\int_{X_s^{k+2}}\sum_{i=0}^{k+1}\phi\bigl((k+2)u(\partial_i\Delta)\bigr)\,d\Delta\\
&\leq \frac{1}{k+2}\sum_{i=0}^{k+1}\int_{X_s^{k+1}}\mu\bigl(B(x_{j_i},s)\bigr)\phi\bigl((k+2)u(\partial_i\Delta)\bigr)\,d\Delta, 
\end{align*}
where $j_i$ is any index different from $i$. Applying \eqref{GeomAcotada}, we get
$$\rho_{\phi,s}(du)\leq V(s)\rho_{\phi,s}\bigl((k+2)u\bigr),$$
and then, by Remark \ref{equivNormas}, $\|du\|_{\phi,s}\preceq \|u\|_{\phi,s}$. 
\end{proof}

We now consider the complex
$$AS^0_\phi(X)\stackrel{d_0}{\rightarrow} AS^1_\phi(X)\stackrel{d_1}{\rightarrow} AS^2_\phi(X)\stackrel{d_2}{\rightarrow} \cdots$$
Its cohomology is the \textit{asymptotic $L^\phi$-cohomology} of $X$. We denote it by
$$L^\phi H_{AS}^
*(X)=\frac{\Ker d_*}{\I d_{*-1}}.$$
We also define the \textit{reduced $L^\phi$-cohomology} of $X$ as the family of Fr\'echet spaces 
$$L^\phi \overline{H}_{AS}^*(X)=\frac{\Ker d_*}{\overline{\I d_{*-1}}}.$$

\subsection{Quasi-isometry invariance}

By a kernel on $X$ we mean a non-negative bounded function $\kappa:X\to\R$ satisfying the following conditions:
\begin{enumerate}
    \item[(c)] There exists $K>0$ such that if $|x-x'|>K$, then $\kappa(x,x')=0$.
    \item[(d)] For every $x\in X$,
    \begin{equation}\label{kernel1}
    \int_X \kappa(x,x') dx'=1.
    \end{equation}
\end{enumerate}
Because of the bounded geometry \eqref{GeomAcotada}, it is always possible to take a kernel on $X$, for instance one can consider
$$\kappa(x,x')=\frac{1}{\mu(B(x,K))}\mathbbm{1}_{B(x,K)}(x'),$$
with $K\geq r_0$.

Observe that for every $x'\in X$,
\begin{equation}\label{kernel2}
\int_X\kappa(x,x')dx\leq \sup (\kappa)V(K).
\end{equation}

If $\Delta=(x_0,\ldots,x_k)$ and $\Delta'=(x'_0,\ldots,x'_k)$, then we write 
\begin{equation}\label{kernel3}
\kappa(\Delta,\Delta')=\prod_{i=0}^k \kappa(x_i,x'_i).    
\end{equation}
It is clear that for a fixed $\Delta=(x_0,\ldots,x_k)$ we have $\int_{X^{k+1}}\kappa(\Delta,\Delta')\,d\Delta'=1$.

Now consider another metric space $(Y,|\cdot-\cdot|)$ equipped with a Borel measure $\nu$  satisfying \eqref{GeomAcotada} with functions $\overline{v}$ and $\overline{V}$. Suppose that $F:X\to Y$ is a quasi-isometry and $\overline{F}:Y\to X$ is a quasi-inverse of $F$. We can assume that $\lambda\geq 1$ and $\epsilon\geq 0$ satisfy conditions $(a)$ and $(b)$ for both $F$ and $\overline{F}$. Observe that $F$ and $\overline{F}$ induce quasi-isometries $F:X^{k+1}\to Y^{k+1}$ and $\overline{F}:Y^{k+1}\to X^{k+1}$ with the same constants.

For a kernel $\kappa_Y$ in $Y$, define the pull-back of a function $u:Y^{k+1}\to \R$ by $F$ as follows:
$$F^* u: X^{k+1}\to \R,\ F^* u(\Delta_X)=\int_{Y^{k+1}} u(\Delta_Y)\kappa_Y(F\Delta_X,\Delta_Y)\, d\Delta_Y.$$ 

\begin{lemma}\label{pull-back}
The pull-back $F^*$ defines a continuous map from $AS_\phi^k(Y)$ to $AS_\phi^k(X)$.
\end{lemma}

\begin{proof}
First observe that Jensen's inequality implies
$$\phi\bigl(F^*u(\Delta_X)\bigr)\leq \int_{Y^{k+1}} \phi\bigl(u(\Delta_Y)\bigr)\kappa_Y(F\Delta_X,\Delta_Y)\, d\Delta_Y.$$
Hence, for every  $s>0$ we have
\begin{align*}
\rho_{\phi,s}\bigl(F^*u(\Delta_X)\bigr) & \leq \int_{X_s^{k+1}} \int_{Y^{k+1}} \phi\left(u(\Delta_Y)\right)\kappa_Y(F\Delta_X,\Delta_Y)\, d\Delta_Yd\Delta_X\\
& = \int_{Y^{k+1}} \phi\left(u(\Delta_Y)\right)\Psi_s(\Delta_Y)\, d\Delta_Y,
\end{align*}
where 
$$\Psi_s(\Delta_Y)=\int_{X_s^{k+1}} \kappa_Y(F\Delta_X,\Delta_Y)\,d\Delta_X.$$

\underline{Claim}: $\Psi_s$ is bounded and its support is included in $Y_{s'}^{k+1}$, where $s'$ depends on $s$, the constants $K$, $\lambda$, and $\epsilon$ and the function $V$.\\

For $\Delta_Y\in Y^{k+1}$ consider $\overline{\Delta}_X\in X^{k+1}$ such that $|F\overline{\Delta}_X-\Delta_Y|\leq \epsilon$ (this simplex exists because $F$ is a quasi-isometry with constants $\lambda$ and $\epsilon$). If $\kappa_Y(F\Delta_X,\Delta_Y)\neq 0$, then $|F\overline{\Delta}_X-F\Delta_X|\leq \epsilon+K$ and therefore $|\overline{\Delta}_X-\Delta_X|\leq \lambda(2\epsilon+K)$. This implies that if $H$ is the supremum of $\kappa_Y$, then
$$\Psi_s(\Delta_Y)\leq\int_{X_s^{k+1}} H \mathds{1}_{B\bigl(\overline{\Delta}_X,\lambda(2\epsilon+K)\bigr)}d\Delta_X\leq H V\bigl(\lambda(2\epsilon+K)\bigr)^{k+1}=:\overline{H},$$
which shows that $\Psi_s$ is bounded.

In order to prove the other part of the claim, observe that if $\Delta_X\in X_s^{k+1}$, then $\mathrm{diam}(F\Delta_X)\leq \lambda s+\epsilon$. If in addition $\kappa_Y(F\Delta_X,\Delta_Y)\neq 0$ for some $\Delta_X\in X_s^{k+1}$, then we have 
\begin{equation}\label{diam-deltay}
\mathrm{diam}(\Delta_Y)\leq 2K+\lambda s+\epsilon.
\end{equation}
Indeed, if $\Delta_X=(x_0,\dots,x_k)$ and $\Delta_Y=(y_0,\dots,y_k)$ then for any $i$ and $j$
$$
|y_i-y_j|\le |y_i- F(x_i)| + |F(x_i)-F(x_j)| + |F(x_j)-y_j| \le  K + \lambda s + \epsilon + K = 2K + \lambda s + \epsilon,
$$
which implies~\eqref{diam-deltay}.

The proof of the claim finishes by taking $s'=2K+\lambda s+\epsilon$.\\

Putting the above together, for $s>0$ and $u\in AS_\phi^{k}(Y)$ we have $\rho_{\phi,s}(F^*u)\leq \overline{H}\rho_{\phi,s'}(u)$, and hence $\|F^*u\|_{\phi,s}\preceq \|u\|_{\phi,s'}$, where the constant does not depend on $u$.
\end{proof}

It is easy to show that the pull-back $F^*$ commutes with $d$, and thus it defines maps in (reduced) cohomology.

\begin{lemma}\label{lemaKernel}
Suppose that $\kappa_X$ and $\kappa_Y$ are kernels on $X$ and $Y$ respectively. Then the function $\kappa:X\times X\to\R$ defined by 
$$\kappa(x,x')=\int_Y \kappa_Y\bigl(F(x),y\bigr)\kappa_X\bigl(\overline{F}(y),x'\bigr)\,dy$$
is a kernel. 
\end{lemma}

\begin{proof}
Suppose that $K>0$ is the constant in $(c)$ for both kernels $\kappa_X$ and $\kappa_Y$. Assume also that the uniform distance between $\overline{F}\circ F$ and $Id_X$ and between $F\circ \overline{F}$ and $Id_Y$ is bounded by $C\geq 0$.

Observe that if $\kappa_Y\bigl(F(x),y\bigr)\neq 0$, then $|x-\overline{F}(y)|\leq \lambda K+\epsilon+C$. From this we conclude that if $|x-x'|>K'=(\lambda+1)K+\epsilon+C$, then $\kappa_Y\bigl(F(x),y\bigr)=0$ or $\kappa_X\bigl(\overline{F}(y),x'\bigr)=0$ for every $y\in Y$, and hence $\kappa(x,x')=0$.

To see that $\kappa$ is bounded observe that given $x,x'\in X$, the support of the function 
$$y\mapsto \kappa_Y\bigl(F(x),y\bigr)\kappa_X\bigl(\overline{F}(y),x'\bigr)$$
is contained in the ball $B(F(x),K)$, thus
$\kappa(x,x')\leq \overline{V}(K)\sup(\kappa_X)\sup(\kappa_Y)$.

A direct calculation shows that $\int_X\kappa(x,x') dx'=1$ for every $x\in X$, which finishes the proof. \end{proof}

In order to prove Theorem \ref{MainInvariance}, we adapt an argument given in \cite{Pa95} (see also \cite{G}). In particular, we use the following operator: given $u:X^{k+2}\to \R$ and $\Delta\in X^{k+1}$, we consider
$$B_k u (\Delta)=\int_{X^{k+1}} u\bigl(b(\Delta,\Delta')\bigr)\kappa(\Delta,\Delta')\,d\Delta',$$
where 
$$b(\Delta,\Delta')=\sum_{i=0}^k(-1)^i (x_0,\ldots,x_i,x'_i,\ldots,x_k').$$
for $\Delta=(x_0,\ldots,x_k)$ and $\Delta'=(x_0',\ldots,x_k')$. Here $\kappa$ is the kernel given by Lemma \ref{lemaKernel}.

\begin{lemma}[Lemma 3.3.3 in \cite{G}]\label{lemaGenton}
Let $\Delta,\Delta'\in X^{k+1}$, then
$$\partial b(\Delta,\Delta')= \Delta'-\Delta -\sum_{i=0}^k b(\partial_i\Delta,\partial_i\Delta').$$
\end{lemma}

\begin{lemma}
For every $k\geq 0$, $B_k$ defines a continuous operator from $AS_\phi^{k+1}(X)$ to $AS_\phi^{k}(X)$.
\end{lemma}

\begin{proof}
Fix $s>0$ and take $u\in AS_\phi^{k+1}(X)$ and $\Delta=(x_0,\ldots,x_k)\in X_s^{k+1}$, then
\begin{align*}
|B_ku (\Delta)| 
\leq \sum_{i=0}^k \int_{X^{k+1}} |u(\Delta_i)|\kappa(\Delta,\Delta')d\Delta',
\end{align*}
where $\Delta'=(x_0',\ldots,x_k')$ and $\Delta_i=(x_0,\ldots,x_i,x_i',\ldots,x_k')$. Using Jensen's inequality, we have
\begin{align*}
\rho_{\phi,s}\left(B_ku\right) \preceq \sum_{i=0}^k \int_{X_s^{k+1}}\int_{X^{k+1}} \phi\bigl(u(\Delta_i)\bigr)\kappa(\Delta,\Delta')\,d\Delta'd\Delta.
\end{align*}
We write each term of the above sum as
$$\int_{X^{k+1}}\int_{X^{k+1}} \phi\bigl(u(\Delta_i)\bigr)\mathds{1}_{X^{k+1}_s}(\Delta)\kappa(\Delta,\Delta')\,d\Delta'd\Delta.$$
Notice that $\mathds{1}_{X^{k+1}_s}(\Delta)\kappa(\Delta,\Delta')\neq 0$ implies $\Delta_i\in X^{k+1}_{s+2K'}$. Thus, using \eqref{kernel1}, \eqref{kernel2} and \eqref{kernel3}, we obtain
$$\rho_{\phi,s}\left(B_ku\right)\preceq \sum_{i=0}^k \int_{X^{k+2}_{s+2K'}}(\sup \kappa)V(K')^k\phi\bigl(u(\Delta_i)\bigr)\,d\Delta_i \preceq \rho_{\phi,s}\left(u\right).$$
This implies that $\|B_k u\|_{\phi,s}\preceq \|u\|_{\phi,s+2K'}$, which finishes the proof.
\end{proof}

\begin{proof}[Proof of Theorem \ref{MainInvariance}]
We need to prove that $F^*\circ\overline{F}^*$ and $\overline{F}^*\circ F^*$ are homotopic to the identity. We will prove the first assertion by verifying
\begin{equation}\label{homotopia}
\left\{\begin{array}{cc} B_0\circ d_0 = F^*\circ \overline{F}^* -Id & \\ B_{k+1}\circ d_{k+1}+d_k\circ B_{k}=F^*\circ \overline{F}^* -Id & \text{ for all }k\geq 0.
\end{array}\right.
\end{equation}
The other part is analogous.

If $u\in AS_\phi^0(X)$, then we have
\begin{align*}
(B_0\circ d_0) u(x_0) &= \int_X du\bigl(b(x_0,x)\bigr)\kappa(x_0,x)\,dx = \int_X u(x)\kappa(x_0,x)\,dx-u(x_0)\\
&= \int_X u(x)\left( \int_Y \kappa_Y\bigl(F(x_0),y\bigr)\kappa_X\bigl(\overline{F}(y),x\bigr)dy \right) dx - u(x_0)\\
&=\int_Y \left( \int_X u(x) \kappa_X\bigl(\overline{F}(y),x\bigr)dx \right) \kappa_Y\bigl(F(x_0),y\bigr)\,dy - u(x_0)\\
 &= (F^*\circ \overline{F}^*) u(x_0)-u(x_0).
\end{align*}
Therefore, $B_0\circ d_0=F^*\circ \overline{F}^* -Id$.

Now we take $u\in AS_\phi^{k+1}(X)$. First observe that
\begin{align*}
(d_k \circ B_k) u(\Delta) &= B_ku(\partial \Delta)=B_ku\left(\sum_{i=0}^k (-1)^i\partial_i \Delta\right)=\sum_{i=0}^k(-1)^i B_ku(\partial_i\Delta)\\
&= \sum_{i=0}^k (-1)^i \int_{X^{k+1}} u\bigl(b(\partial_i\Delta,\Delta')\bigr)\kappa(\partial_i\Delta,\Delta')\,d\Delta'. 
\end{align*}
By Lemma \ref{lemaGenton}, $(B_{k+1}\circ d_{k+1}) u(\Delta)$ is equal to
\begin{align*}
\int_{X^{k+2}} u(\Delta')\kappa(\Delta,\Delta')\,d\Delta'-u(\Delta)-\sum_{i=0}^k(-1)^i\int_{X^{k+2}} u\bigl(b(\partial_i\Delta,\partial_i\Delta' )\bigr)\kappa(\Delta,\Delta')\,d\Delta'.
\end{align*}
As in the case $k=0$, the first term is equal to $(F^*\circ \overline{F}^*)u(\Delta)$. With respect to the third term, for $\Delta=(x_0,\ldots,x_k)$ and $\Delta'=(x_0',\ldots,x_k')$, we have
\begin{align*}
\sum_{i=0}^k(-1)^i\int_{X^{k+2}} & u\bigl(b(\partial_i\Delta,\partial_i\Delta' )\bigr)\kappa(\Delta,\Delta')\,d\Delta'\\
&= \sum_{i=0}^k(-1)^i\int_{X^{k+1}} u\bigl(b(\partial_i\Delta,\partial_i\Delta' )\bigr)\kappa(\partial_i\Delta,\partial_i\Delta')\left(\int_X \kappa(x_i,x_i')dx_i'\right) d(\partial_i\Delta')\\
&= (d_k\circ B_k)u(\Delta)
\end{align*}
This shows that $B_{k+1}\circ d_{k+1}+B_k\circ d_k = F^*\circ \overline{F}^* -Id$ for every $k\geq 0$.
\end{proof}

\begin{remark}\label{obsIndependencia}
Observe that if $F_1,F_2:X\to Y$ are two quasi-isometries at bounded uniform distance, then a quasi-isometry $G:Y\to X$ is a quasi-inverse of $F_1$ if and only if it is a quasi-inverse of $F_2$. We have proven that in this case $F_1^*\circ G^*$ and $F_2^*\circ G^*$ are homotopy equivalent and $G^*$ is invertible. As a consequence, $F_1$ and $F_2$ induce the same isomorphism in (reduced) cohomology.
\end{remark}

\begin{remark}
Theorem \ref{MainInvariance} says that the asymptotic Orlicz cohomology of $(X,\mu)$ does not depend on the measure $\mu$. Thus, one can define such cohomology for any metric space admitting measures with bounded geometry. 

A metric condition that guarantees the existence of such a measure is a weak version of doubling condition for metric spaces: there exists a constant $\epsilon$ and a function $V:(0,+\infty)\to (0,+\infty)$ such that any $\epsilon$-separated set (i.e. set of point at mutual distance at least $\epsilon$) in a ball of radius $r$ cannot contain more than $V(r)$ points. From this condition one can take $\mu$ as the counting measure on a maximal $\epsilon$-separated discrete set in $X$. 

Observe that if $X$ is a doubling metric spaces, the function $V$ can be taken with polynomial growth at $\infty$ (see Sections 1.3.1 and 1.4.1 in \cite{MT}).  
\end{remark}

\section{Continuous group Orlicz cohomology}

Let $G$ be a locally compact second countable group equipped with a Haar measure $\mathcal{H}$ and a left invariant proper metric $|\cdot-\cdot|$. Fix a doubling Young function $\phi$. 

\begin{lemma}\label{Regular}
The right-regular representation of $G$ on $L^\phi (G)=L^\phi(G,\mathcal{H})$, 
$$\bigl(\pi(g)f\bigr)(x)=f(xg)$$
for every $f\in L^\phi (G)$ and $g,x\in G$, is well-defined and continuous.
\end{lemma}

\begin{proof}
First observe that if $g\in G$, then $\rho_\phi\bigl(\pi(g) u\bigr)=\Delta(g) \rho_\phi(u)$, where $\Delta$ is the modular function associated to $\mathcal{H}$. Hence, the representation is well-defined. 

To prove continuity, consider $g_n\to g$ in $G$ and $f_n\to f$ in $L^\phi(G)$. Observe that 
$$\|\pi(g_n)f_n-\pi(g)f\|_\phi\leq \|\pi(g_n)f_n-\pi(g_n)f\|_\phi+\|\pi(g_n)f-\pi(g)f\|_\phi,$$
where, by Proposition \ref{PropDoblante}, the first term of the right-hand side converges to $0$ because
$$\rho_\phi\bigl(\pi(g_n)f_n-\pi(g_n)f\bigr)=\rho_\phi\bigl(\pi(g_n)(f_n- f)\bigr)=\Delta(g_n)\rho_\phi(f_n-f)\to 0.$$

The second term can be bounded as follows:
\begin{equation*}\label{desc}
\|\pi(g_n) f- \pi(g)f\|_\phi\leq \|\pi(g_n)f- \pi(g_n)\tilde{f}\|_\phi+\|\pi(g_n) \tilde{f}- \pi(g)\tilde{f}\|_\phi+\|\pi(g)\tilde{f}- \pi(g) f\|_\phi,   
\end{equation*}
where $\tilde{f}$ is continuous with compact support. By taking $\tilde{f}$ close enough to $f$ (see Lemma \ref{LemaCompSupp}), we can bound the first and third terms on the right-hand side. Moreover, since $\tilde{f}$ is continuous and $g_n\to g$, the sequence of functions
$$x\mapsto \phi\left(\left|\tilde{f}(x g_n)-\tilde{f}(x g)\right|\right)$$
converges pointwise to $0$. If $K\subset G$ is a compact neighborhood of $g$ such that $g_n\in K$ for every $n$, then these functions are bounded by
$$\phi\left(2\max \bigl(\tilde{f}\bigr)\right) \mathbbm{1}_E,$$
with $E=supp(\tilde{f})K^{-1}$. Therefore, the Dominated Convergence Theorem implies that $$\rho_\phi\bigl(\pi(g_n) \tilde{f}-\pi(g)\tilde{f}\bigr)\to 0,$$ and hence, by Proposition \ref{PropDoblante}, $\|\pi(g_n)\tilde{f}-\pi(g)\tilde{f}\|_\phi\to 0$.
\end{proof}

From the right-regular representation $\pi$ we can consider the (reduced) continuous cohomology of $G$ with coefficients in $\bigl(\pi,L^\phi(G)\bigr)$, which we also call \textit{(reduced) continuous $L^\phi$-cohomology of $G$} and denote by $H^*\bigl(G,L^\phi(G)\bigr)$ and $\overline{H}^*\bigl(G,L^\phi(G)\bigr)$.

\begin{remark}
Since $G$ is locally compact and second countable, it can be represented as a union of an increasing sequence of compact subsets $\{K_n\}$. Thus, $C\bigl(G^{k+1},L^\phi(G)\bigr)$ is a Fr\'echet space for the family of semi-norms
$$\|\omega\|_{K_n}=\max\{\|\omega(x_0,\ldots,x_k)\|_{\phi} : x_j\in K_n\text{ for every }j\}.$$
The continuity of the representation $\pi$ implies that $C\bigl(G^{k+1},L^\phi(G)\bigr)^G$ is a closed subspace of $C\bigl(G^{k+1},L^\phi(G)\bigr)$ and hence a Fr\'echet space. We conclude that the reduced cohomology space $\overline{H}^{k+1}\bigl(G,L^\phi(G)\bigr)$ is also a Fr\'echet space. 

\end{remark}

In this section, we prove Theorem \ref{MainEquivalence} by showing that the complex $\Bigl(C\bigl(G^{*+1},L^\phi(G)\bigr)^G,d\Bigr)$ is homotopy equivalent to $\bigl(AS^*_\phi(G),d\bigr)$. For this, we construct a relatively injective strong $G$-resolution of $L^\phi(G)$ such that the associated $G$-invariant complex is homotopy equivalent to the complex $\bigl(AS^*_{\phi}(G),d\bigr)$ and use Propositions \ref{PropSR} and \ref{PropEq}. 

\medskip

Given two proper metric spaces $X$ and $Y$ equipped with Radon 
measures $\mu_X$ and $\mu_Y$ and a doubling Young
function $\phi$, denote by $\mathbb{L}^\phi_{loc}(X,Y)$ the space 
of (classes) of Borel real functions $f$ on $X\times Y$ such that $f|_{K\times Y}\in L^\phi(K\times Y)$
for every compact set $K\subset X$. Endow $\mathbb{L}^\phi_{loc}(X,Y)$ 
with the family of semi-norms
\begin{equation}\label{seminorms}
\|f\|_{\phi,K}=\inf\left\{ \alpha>0 : \rho_{\phi,K}\left(\frac{f}{\alpha}\right)\leq 1  \right\},\ \rho_{\phi,K}\left(f\right)=\int_K\int_Y\phi\left(f\right)\,d\mu_Yd\mu_X,     
\end{equation}
for $K\subset X$ compact. Observe that $\|\ \|_{\phi,K}$ is the norm on the space $L^\phi(K\times Y)$.

Since $X$ is proper, it can be represented as the union of an increasing sequence of compact subsets $K_n$. Thus $\mathbb{L}^\phi_{loc}(X,Y)$ is the inverse limit of the sequence of the Banach spaces $L^\phi(K_n\times Y)$, which implies that it is a Fr\'echet space (using again \cite[Section 3.3.3]{Gol}).



We study in more detail the case where $Y=G$, assuming that $G$ acts on $X$ preserving the measure $\mu_X$.  

\begin{lemma}\label{Aproximacion}
The space of continuous functions with compact support on $X\times G$, denoted by $C_0(X\times G)$, is dense in $\Loc(X,G)$.
\end{lemma}

\begin{proof}
We write $X=\bigcup_{n\in \N} B_n$ with $B_n=B(x_0,n)$. Since $X$ is proper, $\overline{B_n}$ is compact. 


Take $f\in\Loc(X,G)$. Observe that $f|_{B_n\times G}\in L^\phi(B_n\times G)$ for every $n$. Since $C_0(B_n\times G)$ is dense in $L^\phi(B_n\times G)$, for every $n$ we can take $f_n\in C_0(B_n\times G)$ such that 
$$\int_{B_n}\int_G \phi\bigl(|f_n-f|\bigr)\,d\Hh\, d\mu_X\leq \frac{1}{n}.$$
We can extend $f_n$ to the whole $X\times G$ by zero.

Given a compact set $K\subset X$, there exists $n_0$ such that for every $n\geq n_0$ we have $K\subset B_n$, and as a consequence
$$\int_{K}\int_G \phi\bigl(|f_n-f|\bigr)\,d\Hh\, d\mu_X\leq \int_{B_n}\int_G \phi\bigl(|f_n-f|\bigr)\,d\Hh\, d\mu_X\leq \frac{1}{n}.$$
Thus, $\|f_n-f\|_{\phi,K}\to 0$ for every compact set $K\subset X$.
\end{proof}

\begin{lemma}
$\mathbb{L}^\phi_{loc}(X,G)$ is a $G$-module for the representation given by
$$(g\cdot f)(x,h)=f(g^{-1}x,hg).$$
\end{lemma}

\begin{proof} By analogy with the proof of Lemma \ref{Regular}, for every $f\in\mathbb{L}^\phi_{loc}(X,G)$, $g\in G$ and a compact set $K\subset X$, $\rho_{\phi,K}(g\cdot f)=\Delta(g)\rho_{\phi,g^{-1}K}(f)$. Therefore, the representation is well-defined.

Continuity is proven following the argument in the proof of Lemma \ref{Regular}. Indeed, since $G$ and $\mathbb{L}_{loc}^\phi(X,G)$ are both metrizable spaces, it is enough to prove that if $g_n\to g$ in $G$ and $f_n\to f$ in $\mathbb{L}_{loc}^\phi(X,G)$, then $g_n\cdot f_n\to g\cdot f$ in $\mathbb{L}_{loc}^\phi(X,G)$.

Fix a compact set $K\subset X$. Since $\|\ \|_{\phi,K}$ is a semi-norm, we have
\begin{equation}\label{triang}
\|g_n\cdot f_n-g\cdot f\|_{\phi,K}\leq \|g_n\cdot f_n-g_n\cdot f\|_{\phi,K}+\|g_n\cdot f-g\cdot f\|_{\phi,K}.
\end{equation}
Consider a compact neighborhood $V\subset G$ of $g$ and $n_0\in\N$ such that  $g_n\in V$ for every $n\geq n_0$.  The first term of the right-hand side in \eqref{triang} goes to $0$ as $n\to \infty$ because 
$$\rho_{\phi,K}(g_n\cdot f_n-g_n\cdot f)=\Delta(g_n)\rho_{\phi,g_n^{-1}K}(f_n-f)\leq\Delta(g_n)\rho_{\phi,V^{-1}K}(f_n-f)\to 0.$$
Here we use that $\Delta(g_n)\to \Delta(g)<\infty$ and $\rho_{\phi,V^{-1}K}(f_n-f)\to 0$.

The second term can be bounded as follows:
$$\|g_n\cdot f-g\cdot f\|_{\phi,K}\leq \|g_n\cdot f-g_n\cdot \tilde{f}\|_{\phi,K}+\|g_n\cdot \tilde{f}-g\cdot \tilde{f}\|_{\phi,K}+\|g\cdot \tilde{f}-g\cdot  f\|_{\phi,K},$$
where $\tilde{f}\in C_0(X\times G)$.
Taking $\tilde{f}$ closed enough from $f$ (which is possible because of Lemma \ref{Aproximacion}) we can bound the first and third term using the above argument. 

To see that the middle term goes to $0$ as $n\to\infty$, we can proceed in the same way as in Lemma \ref{Regular}. To do that we can take two compact sets $K_1\subset X$ and $K_2\subset G$ such that $supp(\tilde{f})\subset K_1\times K_2$ and dominate the function  
\begin{equation}\label{sequence}
(x,h)\mapsto \phi\left(\bigl|\tilde{f}(g_n^{-1}x,hg_n)-\tilde{f}(g^{-1}x,hg)\bigr|\right)
\end{equation}
by $\phi(2M)\mathbbm{1}_{E}$, where $E=(V K_1)\times (K_2 V^{-1})$. The proof is finished by applying the Dominated Convergence Theorem. 

\end{proof}

From now on we take $X=G^{k+1}$ equipped with the product measure $\mathcal{H}^{k+1}$ and the maximum distance, which are preserved by the action of $G$ by left translations.

In the $L^p$-case the following lemma stems from Theorem 3.4 in \cite{B}. 

\begin{lemma}\label{RelInj}
$\mathbb{L}^\phi_{loc}(G^{n+1},G)$ is a relatively injective $G$-modulus for every $n\geq 0$.
\end{lemma}

For proving Lemma \ref{RelInj}, we need the following lemma:

\begin{lemma}\label{lemita}
Let $G$ be a locally compact group, $X$ a topological space and $x_0\in X$. If $\eta:G\times X\to\R$ is continuous with $\eta(g,x_0)=0$ for every $g\in G$ and $K\subset G$ is compact, then
$$\lim_{x\to x_0}\bigl(\sup\{|\eta(g,x)| : g\in K\}\bigr) =0.$$
\end{lemma}

\begin{proof}
Let $\mathcal{V}$ be the family of neighborhoods of $x_0\in X$. We need to prove that given $\epsilon>0$ there exists $V\in\mathcal V$ such that for every $x\in V$, 
$$\sup\{|\eta(g,x)| : g\in K\}<\epsilon.$$

Suppose this fails, then there exists $\epsilon>0$ such that for every $V\in \mathcal{V}$ there are $x_V\in V$ and $g_V\in K$ with $\eta(g_V,x_V)\geq\epsilon$. Since $K$ is compact, the net $\{g_V\}_{V\in\mathcal{V}}$ has a convergent subnet $\{g_U\}$ to $g\in K$. The net $\{x_U\}$ converges to $x_0$, thus, by continuity of $\eta$, we have $\eta(g,x_0)\geq\epsilon$, which is a contradiction. 
\end{proof}

\begin{proof}[Proof of Lemma \ref{RelInj}]
Let $\iota:A\to B$ be a strong $G$-injection between $G$-modules, $\beta:B\to A$ its left-inverse, and $\varphi:A\to\mathbb{L}^\phi_{loc}(G^{n+1},G)$ a $G$-morphism. We need to prove that there exists a $G$-morphism $\bar{\varphi}:B\to\mathbb{L}^\phi_{loc}(G^{n+1},G)$ such that $\bar{\varphi}\circ \iota =\varphi$.

Let $\chi:G\to\R$ be a non-negative and bounded function with compact support such that 
\begin{equation}\label{chi}
\int_G \chi(g^{-1})\,dg=1.
\end{equation}

If $b\in B$, we define
\begin{equation}
\bar{\varphi}(b) (x_0,\ldots,x_n,x)=\int_G \chi(g^{-1}x_0)\varphi\bigl(\beta(g^{-1}\cdot b)\bigr)(g^{-1}x_0,\ldots,g^{-1}x_n,xg)\,dg.     
\end{equation}
Using that $\iota$ and $\varphi$ are $G$-equivariant, the identity $\beta\circ\iota=\mathrm{Id}_A$ and \eqref{chi}, it is easy to see that $\bar{\varphi}\circ\iota=\varphi$. 

Let us prove now that $\bar{\varphi}$ is $G$-equivariant: for $h\in G$, 
\begin{align*}
\bar{\varphi}(h\cdot b)(x_0,\ldots,x_n,x)=\int_G \chi(g^{-1}x_0)\varphi\bigl(\beta(g^{-1}h\cdot b)\bigr)(g^{-1}x_0,\ldots,g^{-1}x_n,xg)\,dg.
\end{align*}
Putting $h^{-1}g=\tilde{g}$, we have
\begin{align*}
\bar{\varphi}(h\cdot b)(x_0,\ldots,x_n,x)
&= \int_G \chi(\tilde{g}^{-1}h^{-1}x_0)\varphi\bigl(\beta(\tilde{g}^{-1}\cdot b)\bigr)(\tilde{g}^{-1}h^{-1}x_0,\ldots,\tilde{g}^{-1}h^{-1}x_n,xh\tilde{g})\, d\tilde{g}\\
&=\bar{\varphi}(b)(h^{-1}x_0,,\ldots,h^{-1}x_n,xh)=\bigl( h\cdot \bar{\varphi}(b)\bigr)(x_0,\ldots,x_n,x).
\end{align*}


To prove that $\bar{\varphi}$ is continuous, first observe that, by Jensen's inequality, 
\begin{align*}
\phi\bigl(\bar{\varphi}(b)(x_0,\ldots,x_n,x)\bigr)
&\leq \int_G \phi\Bigl(\varphi\bigl(\beta(g^{-1}\cdot b)\bigr)(g^{-1}x_0,\ldots,g^{-1}x_n,xg) \Bigr) \chi(g^{-1}x_0)\, dg.
\end{align*}
Therefore, if $K\subset G^{n+1}$ is a compact subset, we have
\begin{align*}
&\rho_{\phi,K}\bigl(\bar{\varphi}(b)\bigr) =\int_G \int_K   \phi\bigl(\bar{\varphi}(b)(x_0,\ldots,x_n,x)\bigr)\, dx_0\ldots dx_n\, dx \\
&\leq \int_G\int_K\int_G \phi\Bigl(\varphi\bigl(\beta(g^{-1}\cdot b)\bigr)(g^{-1}x_0,\ldots,g^{-1}x_n,xg)\Bigr) \chi(g^{-1}x_0) \,dg\,dx_0\ldots dx_n\,dx.
\end{align*}
Let $K_0\subset G$ be the projection of $K$ on the first coordinate. Observe that if $x_0\in K_0$ and $\chi(g^{-1}x_0)\neq 0$, then $g\in \tilde{K}=K_0\,supp(\chi)^{-1}$, which is a compact set. If $C=\sup(\chi)$, we have 
\begin{align*}
\rho_{\phi,K}(\bar{\varphi}(b)) &\leq C \int_{\tilde{K}}  \int_G\int_K \phi\Bigl(\varphi\bigl(\beta(g^{-1}\cdot b)\bigr)(g^{-1}x_0,\ldots,g^{-1}x_0,xg) \Bigr)\,  dx_0\ldots dx_n\, dx\,  dg\\
& = C\int_{\tilde{K}} \left( \Delta(g) \int_G\int_K \phi\Bigl(\varphi\bigl(\beta(g^{-1}\cdot b)\bigr)(x_0,\ldots,x_n,x) \right)  dx_0\ldots dx_n\,dx  \Bigr)dg\\
&= C\int_{\tilde{K}} \Delta(g)\rho_{\phi,K}\Bigl(\varphi\bigl(\beta (g^{-1}\cdot b)\bigr)\Bigr) dg.
\end{align*}
Since the representations and the maps $\beta$, $\varphi$, $\rho_{\phi,K}$ and $\Delta$ are continuous, the function
$$\Psi_b:G\to\R,\quad  g\mapsto \Delta(g)\rho_{\phi,K}\Bigl(\varphi\bigl(\beta (g^{-1}\cdot b)\bigr)\Bigr),$$
is 
continuous and hence $\rho_{\phi,K}(\bar{\varphi}(b))<+\infty$. We conclude that $\bar{\varphi}(b)\in \mathbb{L}^\phi_{loc}(G^{n+1},G)$.

Moreover, by Lemma \ref{lemita} applied to the function 
$\eta:G\times B\to \R,\ \eta(g,b)=\Psi_b(g)$, we have
$$\rho_{\phi,K}\bigl(\bar{\varphi}(b-b_0)\bigr)\to 0 \text{ as } b\to b_0,$$
which implies that $\|\bar{\varphi}(b)-\bar{\varphi}(b_0)\|_{\phi,K}\to 0$ as $b\to b_0$ because $\phi$ is doubling. Since $K\subset G$ is any compact set, we conclude that $\bar{\varphi}$ is continuous.

\end{proof}

Consider the~complex of~Fr\'echet $G$-modules
\begin{equation}\label{comp-res}
0\rightarrow L^\phi(G) \overset{\delta_{-1}}{\to} \mathbb{L}^\phi_{loc}(G,G)
\overset{\delta_{0}}{\to} \mathbb{L}^\phi_{loc}(G^2,G) 
\overset{\delta_{1}}{\to} \mathbb{L}^\phi_{loc}(G^3,G) \overset{\delta_{2}}{\to} 
\cdots \,,  
\end{equation}
where
$$
(\delta_k f)(x_0,\dots,x_{k+1},g) 
= \sum_{i=0}^{k+1} (-1)^i f(x_0,\dots,\hat{x}_i,\dots,x_{k+1},g).
$$
This complex is an Orlicz version of the resolution $L^p_{loc}(G^{*+1},L^p(G))$ considered in \cite{BR}. In general, Blanc shows in \cite{B} that $L^p_{loc}(G^{*+1},V)$ is a relatively injective strong $G$-resolution for every $G$-module $V$.  

\begin{lemma}\label{dgroups-cont}
For every $k\geq 0$, the operator $\delta=\delta_k$ is a $G$-morphism from
$\mathbb{L}^\phi_{loc}(G^{k+1},G)$ to $\mathbb{L}^\phi_{loc}(G^{k+2},G)$.
Moreover, $\delta^2=0$.
\end{lemma}

\begin{proof}
Take $f\in \mathbb{L}^\phi_{loc}(G^{k+2},G)$. Since every compact set
in~$G^{k+2}$ is contained in~a~compact set of~the~form $K^{k+2}$, where
$K$ is a~compact set in~$G$, it suffices to~estimate
$\rho_{\phi,K^{k+2}}(\delta f)$ for proving that $\delta$ is well-defined and continuous. 

Using Jensen's inequality,
\begin{align*}
\rho_{\phi,K^{k+2}}(\delta f) &\le \frac{1}{k+2} \sum_{i=0}^{k+1} \mathcal{H}(K) 
\int_{G\times K^{k+1}} 
\phi\bigl( (k+2) f(x_0,\dots,\hat{x_i},\dots,x_{k+1},g) \bigr)  dx_0 \dots dx_{k+1}\, dg \\
&\leq \mathcal{H}(K)\rho_{\phi,K^{k+1}}\bigl((k+2) f\bigr)
\end{align*}
Therefore,
$
\|\delta f\|_{\phi,K^{k+2}}
\preceq  \|f\|_{\phi,K^{k+1}}
$.

The $G$-equivariance of~$\delta$ and the~relation $\delta^2=0$ are straightforward
from~the~definition.


\end{proof}

\begin{lemma}\label{strong-res}
The complex \eqref{comp-res} is a~strong resolution of~$L^\phi(G)$.
\end{lemma}

\begin{proof}

By Remark \ref{ResolucionFuerte} it suffices to construct a continuous contracting homotopy, that is, a family of continuous linear maps $\sigma_{k}:\mathbb{L}^\phi_{loc}(G^{k+1},G)\to \mathbb{L}^\phi_{loc}(G^{k},G)$ for $k\geq 1$ and $\sigma_{0}:\mathbb{L}^\phi_{loc}(G,G)\to L^\phi(G)$ such that
\begin{equation}\label{contractante}
\left\{\begin{array}{cc}
    \sigma_0\circ \delta_{-1}=Id & \\
    \sigma_{k+1}\circ\delta_{k}+\delta_{k-1}\circ \sigma_k=Id &  \text{for all }k\geq 0
\end{array}
\right.
\end{equation}

To this end, we begin by considering a non-negative bounded function $\chi:G\to\R$ with compact support $K_\chi$ such that
$$\int_G \chi(x) dx =1.$$
Then, given $f\in \mathbb{L}^\phi_{loc}(G^{k+1},G)$, we define (where it exists) 
\begin{equation}\label{sigma-n}
(\sigma_{k} f)(x_0,\dots,x_{k-1},g) = (-1)^{k}\int_G f(x_0,\dots,x_{k-1},x,g) \chi(x) \,dx.
\end{equation}
In the case $k=0$, the left-hand side of \eqref{sigma-n} is $(\sigma_{0}f)(g)$. 

Let us prove that the~expression  \eqref{sigma-n} is defined for~almost every
$(x_0,\dots,x_{k-1})\in G^{k}$ and $g\in G$. Since 
$f\in \mathbb{L}^\phi_{loc}(G^{k+1},G)$, for~any compact set $K\subset G^{k}$,
we have 
$$
\int_G \int_{K} \int_{K_\chi} \phi\bigl(f(x_0,\dots,x_{k-1},x,g)\bigr)\,
dx\, dx_0\dots dx_{k-1}\, dg < +\infty.
$$
Thus, 
$$
\int_{K_\chi} \phi\bigl( f(x_0,\dots,x_{k-1},x,g)\bigr)\,
dx <+\infty
$$
for~almost every $(x_0,\dots,x_{k-1})\in K$ and $g\in G$, that is, the function $x\mapsto f(x_0,\dots,x_{k-1},x,g)$ belongs to $L^\phi(K_\chi)$ and hence it belongs to $L^1(K_\chi)$ for these values of $(x_0,\ldots,x_{k-1})$ and $g$. This implies that $\sigma_kf$ is well-defined for almost every point in $K\times G$, and since $G^{k}$ can be written as a countable union of compact sets, it is well-defined for almost every point in $G^{k}\times G$. The argument also works in the case $k=0$ by omitting the compact set $K$.


To see that  $\sigma_{k}$ maps continuously 
$\mathbb{L}^\phi_{loc}(G^{k+1},G)$ into $\mathbb{L}^\phi_{loc}(G^{k},G)$, take  $f\in \mathbb{L}^\phi_{loc}(G^{k+1},G)$ and $K$ a compact subset of $G^{k}$, then by Jensen's inequality we have
$$\rho_{\phi,K}(\sigma_kf)\leq \int_G\int_{K\times K_\chi} \frac{\sup\chi}{\mathcal{H}(K_\chi)}\,\phi\bigl( \mathcal{H}(K_\chi)  f(x_0,\dots,x_{k-1},x,g)\bigr)
dx_0\dots dx_{k-1}\, dx\, dg.$$
This implies that $\|\sigma_kf\|_{\phi,K}\preceq \|f\|_{\phi,K\times K_\chi}$. As above, the same argument works for $k=0$.

The lemma is proven with the verification of \eqref{contractante}, which is straightforward from the definitions.  



\end{proof}

As a consequence of Lemmas \ref{RelInj} and \ref{strong-res}, we conclude that $\bigl(\mathbb{L}^\phi_{loc}(G^{*+1},G),\delta\bigr)$ is a relatively injective strong $G$-resolution. By Propositions \ref{PropSR} and \ref{PropEq}, the cohomology of $\bigl(\mathbb{L}^\phi_{loc}(G^{*+1},G)^G,\delta\bigr)$ is isomorphic to the continuous $L^\phi$-cohomology of $G$.

With the following lemma we finish the proof of Theorem \ref{MainEquivalence}.

\begin{lemma}\label{two-com}
For every $k\geq 0$, the spaces $\mathbb{L}^\phi_{loc}(G^{k+1},G)^G$ and $AS^k_\phi(G)$
are isomorphic. Furthermore, the isomorphism commutes with the derivatives  $\delta$ and $d$. 
\end{lemma}

In the proof of this lemma, we will use a proposition from Zimmer's book (\cite[Section B.5]{Z}). A simplified version of it is used in \cite{BP} to prove the $L^p$ version of the lemma.

\begin{proposition}\label{Zimmer}
Let $X$ and $Y$ be two standard Borel spaces (i.e. they are isomorphic to some Borel subset of a complete separable metric space) and $G$ a locally compact second countable group acting on $X$ and $Y$. Suppose that $\mu$ is a $G$-quasi-invariant Borel measure on $X$ (i.e. $\mu(gA)=0$ if and only if $\mu(A)=0$ for any $A\subset X$) and $f:X\to Y$ is a Borel function such that for every $g\in G$ we have $f(gx)=gf(x)$ for $\mu$-almost every $x\in X$. Then there exists a $G$-invariant Borel subset of full measure $X_0\subset X$ and a $G$-equivariant Borel function $\tilde{f}:X_0\to Y$ that coincides with $f$ almost everywhere.
\end{proposition}

\begin{proof}[Proof of Lemma \ref{two-com}]
Given $u\in AS^k_\phi(G)$, define $\Lambda(u):G^{k+1}\times G\to\R$ by
$$\Lambda(u)(x_0,\ldots,x_n,g)=u(gx_0,\ldots,gx_k).$$
Then $\Lambda$ must be a Borel function because $u$ is Borel and the map $\theta:G^{k+1}\times G\to G^{k+1}$,  $\theta(x_0,\ldots,x_k,g)=(gx_0,\ldots,gx_k)$,
is continuous.

Moreover, if $A=u^{-1}(\R\setminus \{0\})$, we have
\begin{align*}
\mathcal{H}^{k+2}(\theta^{-1}A) &= \int_G\int_{G^{k+1}}\mathbbm{1}_{\theta^{-1}A}(x_0,\ldots,x_k,g)\,dx_0\ldots dx_k\, dg  \\
&= \int_G\int_{G^{k+1}}\mathbbm{1}_{A}(gx_0,\ldots,gx_k)\,dx_0\ldots dx_k\,dg\\
&=\int_G \mathcal{H}^{k+1}(g^{-1}A)\,dg =\int_G \mathcal{H}^{k+1}(A)\,dg,
\end{align*} 
which implies that $\mathcal{H}^{k+2}(\theta^{-1}A)=0$ if and only if $\mathcal{H}^{k+1}(A)=0$, or, equivalently, $\Lambda(u)=0$ almost everywhere if and only if $u=0$ almost everywhere. As a consequence, the function $u\mapsto\Lambda(u)$ is well-defined and injective to the space of Borel functions up to almost everywhere zero functions. We have to prove that it is well-defined, surjective, continuous, and open from $AS^k_\phi(G)$ to $\mathbbm{L}_{loc}^\phi(G^{k+1},G)^G$. 



It is easy to see that $\Lambda(u)$ is $G$-invariant. Observe that the topology of $\mathbb{L}_{loc}^\phi(G^{k+1},G)$ is generated by the family of semi-norms of the form $\|\ \|_{\phi,K_s^Q}$, where 
$$K_s^Q = \{ (x_0,\dots,x_k)\in G^{k+1}_s \,:\, x_0\in Q \}
$$
for $Q$ any compact set and $s>0$. If $u\in AS_{\phi}^{k}(G)$, then
\begin{align*}
\rho_{\phi,K^Q_s}  \bigl(\Lambda(u)\bigr) 
&= \int_{G}\int_Q\int_{G^{k}} \phi\bigl(u(gx,gxy_1,\ldots,gxy_k)\bigr)\, \mathbbm{1}_{G_s^{k+1}}(1,y_1,\ldots,y_k)  \, dy_1\ldots dy_k\, dx\,  dg\\
&= \int_{G}\int_Q\int_{G^{k}} \Delta(x) \phi\bigl(u(g,gy_1,\ldots,gy_k)\bigr)\, \mathbbm{1}_{G_s^{k+1}}(1,y_1,\ldots,y_k)\, dy_1\ldots dy_k\, dx\, dg\\
&= \left(\int_Q\Delta(x)dx\right) \int_{G^{k+1}_s}\phi\bigl(u(y_0,\ldots,y_k)\bigr)\, dy_0\ldots dy_k
 = \mathcal{D}(Q)\rho_{\phi,s}(u),
\end{align*}
where $\mathcal{D}(Q)=\int_Q\Delta(x)dx<+\infty$. Hence, $\Lambda$ is a well-defined continuous embedding. 

Finally, let us prove that $\Lambda$ is surjective. Take $f\in \mathbb{L}^\phi_{loc}(G^{k+1},G)^G$ and find $u\in AS^{\phi,k}(G)$
with $\Lambda(u)=f$.

We use Proposition \ref{Zimmer} for $X=G^{k+1}\times G$ equipped with the measure $\mathcal{H}^{k+2}$ where $G$ acts by $h\cdot (x_0,\ldots,x_k,g)=(h^{-1}x_0,\ldots,h^{-1}x_k,gh)$, and $Y=\R$ where $G$ acts trivially. Then we obtain a $G$-invariant Borel set of full measure $X_0\subset X$ and a $G$-invariant function $\tilde{f}:X_0\to \R$ that coincides with $f$ almost everywhere. 

Consider for $h\in G$ the set
$$Z_g=\{(x_0,\ldots,x_k)\in G^{k+1} : (x_0,\ldots,x_k,g)\in X_0\}.$$
One can easily verify that $h^{-1}Z_g=Z_{gh}$ for every $g,h\in G$. Thus, an argument as in the beginning of the proof allows to show that $Z_g$ has full measure in $G^{k+1}$.

Define $u:G^{k+1}\to\R$ by $$
u(x_0,\ldots,x_k) =
\left\{
\begin{array}{cc}
\tilde{f}(x_0,\ldots,x_k,1) & \text{ if }(x_0,\ldots,x_k)\in Z_1, \\
0                   & \text{otherwise}.
\end{array}\right. 
$$
To see that $\Lambda(u)=f$, observe that, by definition, for every $g\in G$ and every $(x_0,\ldots,x_k)$ in $Z_g$, 
$$u(gx_0,\ldots,gx_k)=\tilde{f}(gx_0,\ldots,gx_k,1)=\tilde{f}(x_0,\ldots,x_k,g).$$
Therefore, $\Lambda(u)(x_0,\ldots,x_k,g)=f(x_0,\ldots,x_k,g)$ for almost every $(x_0,\ldots,x_k,g)\!\in\! G^{k+1}\!\times G$,  
which finishes the proof.
\end{proof}

\section{The discrete case}

Suppose that $X$ is a finite-dimensional simplicial complex equipped with a length metric of bounded geometry, that is, there exist a constant $C\geq 0$ and an increasing function $N:(0,+\infty)\to (0,+\infty)$ such that
\begin{enumerate}
    \item[(e)] the diameter of every simplex is bounded by $C$;
    \item[(f)] for every $r>0$ the number of simplices that intersect any ball of radius $r$ is bounded by $N(r)$. 
\end{enumerate}

Consider the cochain complex
$$\ell^\phi\bigl(X^{(0)}\bigr)\stackrel{d_0}{\rightarrow} \ell^\phi \bigl(X^{(1)}\bigr)\stackrel{d_1}{\rightarrow} \ell^\phi \bigl(X^{(2)}\bigr)\stackrel{d_2}{\rightarrow}\ldots,$$
where $X^{(k)}$ is the set of $k$-simplices in $X$ and $d=d_k$ is defined by \eqref{DerivadaFormal}, which coincides with the usual co-boundary operator, that is, $d\theta(\sigma)=\theta (\partial\sigma)$ for every $\theta\in \ell^\phi\bigl(X^{(k)}\bigr)$ and $\sigma\in X^{(k+1)}$.
The spaces $\ell^\phi\bigl(X^{(k)}\bigr)$ are Banach spaces equipped with the Luxemburg norm $\|\ \|_{\phi}$. We define the $k$\textit{th (reduced) $\ell^\phi$-cohomology space} of $X$ as
$$\ell^\phi H^k(X)=\frac{\mathrm{Ker}\ d_k}{\mathrm{Im}\ d_{k-1}}\quad \left(\ell^\phi \overline{H}^k(X)=\frac{\mathrm{Ker}\ d_k}{\overline{\mathrm{Im}\ d_{k-1}}}\right).$$

\begin{theorem}[\cite{C}]\label{teoCar}
Let $X$ and $Y$ be two uniformly contractible simplicial complexes with bounded geometry. If they are quasi-isometric, then $\Bigl(\ell^\phi\bigl(X^{(*)}\bigr),d\Bigr)$ and $\Bigl(\ell^\phi\bigl(Y^{(*)}\bigr),d\Bigr)$ are homotopy equivalent for any Young function $\phi$. Hence, their (reduced) cohomologies are isomorphic. 
\end{theorem}
A metric space $X$ is \textit{uniformly contractible} if there exists an increasing function $\varphi:(0,+\infty)\to (0,+\infty)$ such that every ball $B(x,r)$ is contractible in the ball $B(x,\varphi(r))$.

\medskip

Let $G$ be a discrete group acting properly discontinuously, cocompactly, and freely on a contractible locally finite simplicial complex $X$ by simplicial automorphisms. For each $k\geq 0$, consider the space $C\bigl(X^{(k)},\ell^\phi(G)\bigr)$  of functions $f:X^{(k)}\to \ell^\phi(G)$. We equip it with the compact-open topology, which coincides with the topology of pointwise convergence. They are $G$-modules for the action
$$(g\cdot f)(\sigma)=\pi(g)\bigl(f(g^{-1}\sigma)\bigr)\in\ell^\phi(G),\quad g\in G,\ \sigma\in X^{(k)}.$$
Recall that $\pi$ is the right regular representation on $\ell^\phi(G)$.

The derivative $d_k:C\bigl(X^{(k)},\ell^\phi(G)\bigr)\to C\bigl(X^{(k+1)},\ell^\phi(G)\bigr)$ is defined by \eqref{DerivadaFormal}. It is easy to see that it is a $G$-morphism. Then 
$$0\to V\stackrel{d_{-1}}{\rightarrow}C\bigl(X^{(0)},\ell^\phi(G)\bigr)\stackrel{d_0}{\rightarrow}C\bigl(X^{(1)},\ell^\phi(G)\bigr)\stackrel{d_1}{\rightarrow}C\bigl(X^{(2)},\ell^\phi(G)\bigr)\stackrel{d_2}{\rightarrow}\cdots$$
is a relatively injective strong $G$-resolution of $V$ (See \cite[Example 2.2]{BR}).

\begin{proposition}\label{equivSG}
The complexes $\Bigl(C\bigl(G^{*+1},\ell^\phi(G)\bigr)^G,d\Bigr)$ and $\bigl(\ell^\phi \bigl(X^{(*)}\bigr),d\bigr)$ are homotopy equivalent. Thus their (reduced) cohomology are isomorphic.
\end{proposition}

The proof of this proposition is a general version of the proof of Proposition 3.2 in \cite{BR}.

\begin{proof}
By Corollary \ref{corRes}, it suffices to prove that the complex $C\bigl(X^{(k)},\ell^\phi(G)\bigr)^G$ is isomorphic to $\ell^\phi \bigl(X^{(*)}\bigr)$ and the isomorphism commutes with the derivative. To this end, we define $\Psi:\ell^\phi \bigl(X^{(k)}\bigr)\to C\bigl(X^{(k)},\ell^\phi(G)\bigr)$ by
$$\Psi(\theta)=f,\ f(\sigma)(g)=\theta(g\sigma).$$

Observe that, if $\sigma\in X^{(k)}$, then
\begin{equation*}
\rho_\phi\bigl(f(\sigma)\bigr) = \sum_{g\in G} \phi\bigl(\theta(g\sigma))\bigr)\leq \sum_{\sigma\in X^{(k)}} \phi\bigl(\theta(\sigma)\bigr)=\rho_\phi(\theta),
\end{equation*}
where the inequality comes from the fact that $G$ acts freely on $X$. We conclude that $f(\sigma)\in \ell^\phi(G)$ and hence $\Psi$ is well-defined. This also shows that $\Psi$ is continuous, because if $\theta_n\to 0$ in $\ell^\phi(X^{(k)})$, then $\rho_\phi\bigl(\Psi(\theta_n)(\sigma)\bigr)\leq \rho_\phi(\theta_n)\to 0$ for every $\sigma\in X^{(k)}$.

It is easy to see that $\Psi$ is injective, indeed, if $\Psi(\theta)=f=0$, then $\theta(\sigma)=f(\sigma)(1)=0$ for every $\sigma\in X^{(k)}$; and that the image of $\Psi$ is in $C\bigl(X^{(k)},\ell^\phi(G)\bigr)^G$: if $\Psi(\theta)=f$, then
\begin{align*}
(g\cdot f)(\sigma) (h) =f(g^{-1}\sigma)(hg)
= \theta(hgg^{-1}\sigma) =  f(\sigma) (h).
\end{align*}





Now, for $f\in C^k\bigl(X,\ell^\phi(G)\bigr)^G$ define $\theta:X^{(k)}\to \R$ by $\theta(\sigma)=f(\sigma)(1)$. Since $f$ is $G$-invariant, $\theta(g\sigma)=f(\sigma)(g)$ for every $\sigma\in X^{(k)}$ and $g\in G$, which means that $\Psi(\theta)=f$. Moreover, if $A^{(k)}\subset X^{(k)}$ is the (finite) set of $k$-simplices that intersect a compact fundamental domain for the action of $G$, we have 
$$\rho_\phi(\theta)=\sum_{\sigma\in X^{(k)}} \phi\bigl(\theta(\sigma)\bigr)\leq \sum_{\sigma\in A^{(k)}}\sum_{g\in G} \phi\bigl(\theta(g\sigma)\bigr)=\sum_{\sigma\in A^{(k)}}\sum_{g\in G} \phi\bigl(f(\sigma)(g)\bigr)= \sum_{\sigma\in A^{(k)}}\rho_\phi\bigl(f(\sigma)\bigr)
.$$
This shows that the inverse of $\Psi$ is continuous, because if $f_n=\Psi(\theta_n)\to 0$ pointwise, then $\theta_n\to 0$ in $\ell^\phi\bigl(X^{(k)}\bigr)$.




\end{proof}

\begin{remark}
(1) Suppose that the groups $G$ and $G'$ act, in addition, by isometries on $X$ and $X'$ respectively, which are uniformly contractible simplicial complexes with bounded geometry. By  \cite[p. 140, Proposition 8.19]{BH}, $G$ and $G'$ are finitely generated and  quasi-isometric to $X$ and $X'$ respectively when we equip them with word metrics. Combining Theorem \ref{teoCar} and Proposition \ref{equivSG}, we conclude that if $G$ and $G'$ are quasi-isometric, then they have the same (reduced) continuous $L^\phi$-cohomology for any Young function $\phi$.

(2) If $G$ is a finitely generated group, then it acts by isometries (and simplicial isomporphisms) on its Cayley graph $\mathrm{Cay}(G,S)$ for some finite generator $S$. This action is properly discontinuous, free, and cocompact. In general, the Cayley graph is not uniformly contractible; however, if $G$ is in addition a hyperbolic group, then the $n$th Rips complex of $\mathrm{Cay}(G,S)$ is a uniformly contractible simplicial complex (see \cite[p. 469, Proposition 3.23]{BH}) and the action of $G$ on it satisfies the conditions required in Proposition \ref{equivSG}. 
\end{remark}



\section{The case of degree 1}

Inspired by previous works as \cite{C,KP2,MV,Puls,T}, here we study some peculiarities of the case of degree 1.

Let us start with the asymptotic Orlicz cohomology. We assume that $(X,\mu)$ is a metric measure space with bounded geometry and also that $X$ has the \textit{midpoint property}, that is, there is a constant $\mathfrak{c}\geq 0$ such that for any $x,y\in X$ there exists $z\in X$ such that 
$$|x-z|,|y-z|\leq \frac{1}{2}|x-y|+\mathfrak{c}.$$

We denote by $Z_\phi^1(X)$ the kernel of $d:AS_\phi^1(X)\to AS_\phi^2(X)$, then  
$$L ^\phi H_{AS}^1(X)=Z_\phi^1(X)/dL^\phi(X) \text{ and } L ^\phi \overline{H}_{AS}^1(X)=Z_\phi^1(X)/\overline{dL^\phi(X)}.$$ 
Recall that the topology of $AS^k_\phi(X)$ is given by the family of semi-norms \eqref{semi-norms}.

Observe that we can describe $Z^1_\phi(X)$ as the space of classes of functions $u\in AS_\phi^1(X)$ such that 
\begin{equation}\label{condCociclo}
u(x,y)=u(z,y)-u(z,x)    
\end{equation}
for almost all $x,y,z\in X$. This implies that there exists a fixed $z_0\in X$ such that $u(x,y)=u(z_0,y)-u(z_0,x)$ for almost all $x,y\in X$. Thus we can define 
\begin{equation}
f_u(x)=u(z_0,x).
\end{equation}
We have that $df_u(x,y)$ coincides with $u(x,y)$ for almost all $x,y\in X$ and satisfies \eqref{condCociclo} for all $x,y,z\in X$. 

\begin{lemma}\label{lemaNormasEq}
There exists $t_0\geq 0$ such that the semi-norms $\|\ \|_{\phi,t_1}$ and $\|\ \|_{\phi,t_2}$ are equivalent in $Z_\phi^1(X)$ for all $t_1,t_2> t_0$. In particular $\|\ \|_{\phi,t}$ is a norm in $Z_\phi^1(X)$ for every $t> t_0$ and $\bigl(Z_\phi^1(X),\|\ \|_{\phi,t}\bigr)$ is a Banach space.
\end{lemma}

\begin{proof}
Let $v$, $V$ and $r_0$ as in \eqref{GeomAcotada} for the space $(X,\mu)$.

Take $u\in Z^1_\phi(X)$. Because of the above observation, we can suppose that $u$ satisfies \eqref{condCociclo} for every $x,y,z\in X$. We will prove that if $t>t_0:=\max\{8\mathfrak{c},8r_0\}$, then $\|u\|_{\phi,\frac{3t}{2}}\preceq \|u\|_{\phi,t}$, where the constant depends only on $t$. Observe that this proves the first part of the lemma because it is clear that if $t\leq t'$, then $\|u\|_{\phi,t}\leq \|u\|_{\phi,t'}$.

\underline{Claim}: $$\mathbbm{1}_{X^2_{\frac{3t}{2}}}(x,y)\leq \frac{1}{v(t/8)}\int_X \mathbbm{1}_{X^2_t}(x,z)\,\mathbbm{1}_{X^2_t}(z,y)\, dz.$$

If $\mathbbm{1}_{X^2_{\frac{3t}{2}}}(x,y)=1$, take $z_0\in X$ such that $|x-z_0|,|y-z_0|\leq \frac{1}{2}|x-y|+\mathfrak{c}\leq \frac{3t}{4}+\mathfrak{c}$. By the choice of $t_0$, the ball $B(z_0,t/8)$ is included in $B(x,t)\cap B(y,t)$. This implies that
$$\int_X \mathbbm{1}_{X^2_t}(x,z)\,\mathbbm{1}_{X^2_t}(z,y)\, dz\geq \mu\bigl(B(z_0,t/8)\bigr)\geq v(t/8)>0,$$
which proves the claim.

\medskip

Using the claim and Jensen's inequality, we obtain 
\begin{align*}
\rho_{\phi,\frac{3t}{2}}(u) &=\int_{X^2}\phi\bigl(u(x,y)\bigr)\,\mathbbm{1}_{X^2_{\frac{3t}{2}}}(x,y)\,dx\,dy\\
&\leq \frac{1}{v(t/8)} \int_{X^2}\phi\bigl(u(x,z)+u(z,y)\bigr)\left(\int_X \mathbbm{1}_{X^2_t}(x,z)\,\mathbbm{1}_{X^2_t}(z,y) \,dz\right)\,dx\,dy\\
&\leq \frac{1}{2v(t/8)} \int_{X^3}\Bigl(\phi\bigl(2u(x,y)\bigr)+\phi\bigl(2u(x,y)\bigr)\Bigr)\,\mathbbm{1}_{X^2_t}(x,z)\,\mathbbm{1}_{X^2_t}(z,y)\, dx\,dy\,dz\\
&=\frac{V(t)}{v(t/8)}\int_{X^2}\phi\bigl(2u(x,y)\bigr)\,\mathbbm{1}_{X^2_t}(x,z)\,dx\,dz = \frac{V(t)}{v(t/8)} \rho_{\phi,t}\left(2u\right).
\end{align*}
This implies that $$\|u\|_{\phi,\frac{3t}{2}}\leq \frac{2V(t)}{v(t/8)}\|u\|_{\phi,t}.$$

Observe that tor every $t>t_0$ and $u\in Z_\phi^1(X)$, $\|u\|_{\phi,t}=0$ implies that $\|u\|_{\phi,t'}=0$ for any other $t'>t_0$. Hence $u=0$ almost everywhere and, as a consequence, $\|\ \|_{\phi,t}$ is a norm on $Z_\phi^1(X)$.

Finally, since $d$ is continuous, $Z_\phi^1(X)$ is a Fr\'echet space with the topology of $AS_\phi^1(X)$. In addition, the equivalence of the norms $\|\ \|_{\phi,t}$ for $t>t_0$ implies that Cauchy  (resp. convergent) sequences for one of such $t$ are Cauchy (resp. convergent) for every $t'>t_0$. From this we conclude that $\bigl(Z^1_\phi(X),\|\ \|_{\phi,t}\bigr)$ is a Banach space. 
\end{proof}

\begin{remark}
Suppose that $X$ has a length metric (then $\mathfrak{c}=0$) and satisfies $v(t)>0$ for every $t>0$; thus $t_0$ can be taken equal to $0$. 

\end{remark}

Notice that Lemma \ref{lemaNormasEq} implies that $L^\phi \overline{H}^1_{AS}(X)$ is indeed a Banach space if $X$ has the midpoint property.

Now we consider $G$ a locally compact second countable topological group and $\Hh$ a left Haar measure on $G$. We also assume that $G$ is equipped with a left-invariant metric possessing the midpoint property and has a compact generator $S$ that contains an open neighborhood of $1\in G$ that is also a generator. 

We consider the Luxemburg semi-norm $\|\ \|_{\phi,S}$  associated to the modular
$$\rho_{\phi,S}(f)=\int_S\rho_\phi\bigl(\pi(s)f-f\bigr)\,ds = \int_S\int_G \phi\bigl(f(xs)-f(x)\bigr)\,dx\,ds,$$
that is,
$$\|f\|_{\phi,S}=\inf\left\{\alpha>0 : \rho_{\phi,S}\left(\frac{f}{\alpha}\right)\leq 1\right\}.$$

We say that $f:G\to\R$ is a $\phi$\textit{-Dirichlet function} if $\|f\|_{\phi,S}<+\infty$, and take $D_\phi(G)$ the space of classes of $\phi$-Dirichlet functions coinciding almost everywhere. Observe that  $L^\phi(G)$ is contained in $D_\phi(G)$.

Suppose that $t_0\geq 0$ is as in Lemma \ref{lemaNormasEq} and $S$ contains a closed ball $\overline{B}(1,t)$ for some $t>t_0$. Given $f\in D_\phi(G)$ we have
\begin{align*}
\rho_{\phi,t}(df) &=\int_G\int_{\overline{B}(x,t)} \phi\bigl(f(y)-f(x)\bigr)\,dy\, dx =\int_G \int_{\overline{B}(1,t)} \phi\bigl(f(xs)-f(x)\bigl)\,ds\, dx\\
&\leq \rho_{\phi,S}(f) \leq \int_G \int_{\overline{B}(1,t')} \phi\bigl(f(xs)-f(x)\bigl)\,ds\, dx=\rho_{\phi,t'}(df),
\end{align*}
where $t'=\mathrm{diam}(S)$.

The above estimate implies that if $S$ is big enough, then  $d:D_\phi(G)\to Z_\phi^1(G)$ is well-defined and continuous, and its kernel is the subspace of almost everywhere constant functions. Moreover, the induced map $d:\Dd_\phi(G)\to Z_\phi^1(G)$ is a topological embedding, where $\Dd_\phi(G)=D_\phi(G)/\R$ (where $\R$ denotes the subspace of functions constant almost everywhere). We also know that $d$ is surjective because $df_u=u$ for every $u\in Z_\phi^1(G)$, so the map is indeed a topological isomorphism (in particular $\mathcal{D}_\phi(G)$ is a Banach space). We conclude that
\begin{equation}\label{Identificacion}
L^\phi H_{AS}^1(G)\simeq \Dd_\phi(G)/\mathcal{L}^\phi(G)\text{ and }L^\phi \overline{H}_{AS}^1(G)\simeq \Dd_\phi(G)/\overline{\mathcal{L}^\phi(G)},  
\end{equation}
where $\mathcal{L}^\phi(G)$ is the image of $L^\phi(G)$ by the projection $D_\phi(G)\to\Dd_\phi(G)$ (observe that, if $\mu$ is infinite, then $\mathcal{L}^\phi(G)$ coincides with $L^\phi(G)$).
Recall that, if $\phi$ is doubling, then these quotients are isomorphic to $L^\phi H^1\bigl(G,L^\phi(G)\bigr)$ and $L^\phi \overline{H}^1\bigl(G,L^\phi(G)\bigr)$ respectively.  

\begin{remark}\label{generalizacion}
In general, if $(X,\mu)$ is a measure metric space with the midpoint property and $t$ is large enough, then the Banach space $\bigl(Z^1_\phi(X),\|\ \|_{\phi,t}\bigr)$ is isometric to
$$\mathcal{D}_\phi(X)=D_\phi(X)/\R=\left\{f:X\to \R : \|df\|_{\phi,t}<+\infty\right\}/\R,$$
equipped with the natural norm. Therefore, equivalences \eqref{Identificacion} hold for metric spaces. 
\end{remark}

Let us now show an alternative definition of the continuous Orlicz cohomology of a topological group. It comes from the definition of group cohomology in terms of inhomogeneus cocycles (see for example \cite[p. 17]{Gui}). Let $\mathcal{Z}_\phi(G)$ be the space of continuous functions $\omega:G\to L^\phi(G)$ such that $\omega(gh)=\pi(g)\omega(h)+\omega(g)$
for all $g,h\in G$ equipped with the compact-open topology. We also take $\mathcal{B}_\phi(G)$ as the subspace of those functions that can be written as $\omega(g)=\pi(g)f-f$ for some $f\in L^\phi(G)$.

From now on we assume that $\phi$ is doubling. In this case, by Lemma~\ref{Regular}, the elements of $\mathcal{B}_\phi(G)$ are continuous cocycles and so $\mathcal{B}_\phi(G)\subset\mathcal{Z}_\phi(G)$. One can easily verify that the function
$$\mathcal{Z}_\phi(G)\to C\bigl(G^2,L^\phi(G)\bigr),\ \omega\mapsto d\omega,$$
induces isomorphisms 
$$H^1\bigl(G,L^\phi(G)\bigr)\simeq \mathcal{Z}_\phi(G)/\mathcal{B}_\phi(G)\text{ and }\overline{H}^1\bigl(G,L^\phi(G)\bigr)\simeq \mathcal{Z}_\phi(G)/\overline{\mathcal{B}_\phi(G)}.$$

In these equivalences, the group $G$ need not have a compact generator. If, in addition, $G$ has a compact generator $S$ containing an open neighborhood of $1\in G$ that is also a generator, we can define on  $\mathcal{Z}_\phi(G)$ the semi-norm $$\|\omega\|_S=\sup_{s\in S} \|\omega(s)\|_\phi.$$
Observe that, since the modular function $\Delta$ is continuous and $S$ is compact, there exists a constant $M\geq 1$ such that for every $s\in S$ and $f\in L^\phi(G)$,
\begin{equation}\label{estimacion}
\|\pi(s)f\|_\phi\leq M\|f\|_\phi,  
\end{equation}
This observation and the condition $\omega(gh)=\pi(g)\omega(h)+\omega(g)$ imply that $\|\omega\|=0$ if and only if $\omega=0$ (and as a consequence $\|\ \|_S$ is a norm).

\begin{proposition}\label{propZ} 
\begin{enumerate}
    \item[$(i)$] The norm $\|\ \|_{S}$ induces the compact-open topology on $\mathcal{Z}_\phi(G)$.
    \item[$(ii)$] $\bigl(\mathcal{Z}_\phi(G),\|\ \|_S\bigr)$ is a Banach space.
\end{enumerate}
\end{proposition}

\begin{proof}
Since $L^\phi(G)$ is a metric space and $G$ is the union of countably many compact subsets, the compact-open topology is the topology of uniform convergence on compact sets and $G$ is first countable. Suppose that $\omega_n\to 0$ uniformly on compact sets, then for every $\epsilon>0$ there exists $n_0$ such that $\|\omega_n(s)\|_\phi<\epsilon$ for every $s\in S$ and $n\geq n_0$. This implies that $\|\omega_n\|_S\to 0$.

Conversely, suppose that $\|\omega_n\|_S\to 0$ and fix a compact set $K\subset G$. Since $S$ contains an open generator, there exists $k$ such that $K\subset S^k$.
For every $x\in K$ we can write $x=s_1\ldots s_\ell$ with $\ell\leq k$ and $s_1,\ldots,s_\ell\in S$. The condition $\omega(gh)=\pi(g)\omega(h)+\omega(g)$ implies
$$\omega_n(x)=\sum_{i=1}^\ell \pi(s_1\ldots s_{i-1})\omega_{n}(s_i).$$
Therefore, using \eqref{estimacion}, we obtain $\|\omega_n(x)\|_\phi\leq kM^k\|\omega_n\|_S$. Since $\|\omega_n\|_S\to 0$, we have $\|\omega_n(x)\|_\phi\to 0$ uniformly on $K$. This proves $(i)$.

To prove $(ii)$ it is enough to observe that  $\bigl(\mathcal{Z}_\phi(G),\|\ \|_S\bigr)$ can be seen as a closed subspace of $\Bigl(C\bigl(S,L^\phi(G)\bigr),\|\ \|_\infty\Bigr)$. 
\end{proof}

\subsection{$\phi$-Harmonic functions}\label{SectionHarmonic}

Here we assume in addition that $G$ is unimodular and has a compact generator $S$ that is also a symmetric neighborhood of $1\in G$. We also assume that the Haar measure on $G$ is \textit{locally doubling}, that is, for every $R>0$ there exists a constant $C=C(R)\geq 1$ such that for every $x\in G$ and $0<r<R$,
\begin{equation}\label{MedidaDoblante}
0< \mu\bigl(B(x,2r)\bigr)\leq C\mu\bigl(B(x,r)\bigr)<+\infty.
\end{equation}
Throughout this section, $\phi$ will be a doubling strictly convex $N$-function whose derivative exists at every point different from $0$. We extend the derivative $\phi'$ to the whole $\R$ by putting $\phi'(0)=0$. Let $\psi$ be the convex conjugate of $\phi$, which is also an $N$-function in this case. Since $\phi$ is an $N$-function, the function $\eta(s)=\phi'(t)s-\phi(s)$ has a positive maximum for any fixed $t>0$, which is attained at some $s$ such that $\eta'(s)=\phi'(t)-\phi'(s)=0$. Hence $t=s$ because $\phi$ is strictly convex. By the definition of $\psi$ we conclude that 
\begin{equation}\label{Young2}
\psi\bigl(\phi'(t)\bigr)= t\phi'(t)+\phi(t),  
\end{equation}
for every $t>0$. If $t=0$, then the previous equality is obviously true, and since $\phi'$ is an odd function, it also holds for $t<0$.

\begin{lemma}\label{LemmaConjugada}
If $f\in L^\phi(G)$, then $\phi'(f)\in L^\psi(G)$. In fact, 
$\rho_\psi\bigl(\phi'(f)\bigr)\leq (D-1)\rho_\phi(f)$,
where $D$ is a constant satisfying \eqref{ConstDoblante}.
\end{lemma}

\begin{proof}
Since $\phi'$ is non-decreasing we have that for every $t\geq 0$,
$$t\phi'(t)\leq \int_t^{2t}\phi(t)\,dt\leq \int_0^{2t}\phi(t)\,dt=\phi(2t)\leq D\phi(t).$$
This is also true for $t<0$ because $\phi$ is even and $\phi'$ is odd. Using \eqref{Young2} and the previous estimate, we obtain $\psi\bigl(\phi'(f)\bigr)\leq (D-1)\phi(f)$, thus
$$\rho_\psi\bigl(\phi'(f)\bigr)= \int_G \psi\bigl(\phi'(f)\bigr)\,d\mu\leq (D-1)\int_G\phi(f)\,d\mu=\rho_\phi(f) <+\infty.$$
\end{proof}

Define the \textit{$\phi$-Laplacian} of a function $f\in D_\phi(G)$ by
$$\Delta_\phi f(x)=\int_S \phi'\bigl(f(xs)-f(x)\bigr)\,ds.$$
We say that $f$ is $\phi$\textit{-harmonic} if $\Delta_\phi f=0$ almost everywhere. An element $[f]$ of $\mathcal{D}_\phi(G)$ is \textit{$\phi$-harmonic} if $f$ is a $\phi$\textit{-harmonic} function (it does not depend on the representative). Observe that this definition depends on $S$; however, we will see that there exists a one-to-one correspondence between $\phi$-harmonic classes of functions for different generators.

\begin{proposition}
Let $f\in D_\phi(G)$. Then $\Delta_\phi f$ is well-defined and locally integrable.
\end{proposition}

\begin{proof}
Consider $K\subset G$ a compact set with $\mu(K)>0$. Using Tonelli's theorem, Jensen's inequality, and Lemma \ref{LemmaConjugada}, we get
\begin{align*}
\int_{K}\int_S \Bigl|\phi'\bigl(f(xs)-f(x)\bigr)\Bigr|\,ds\,dx
&\leq \mu(K)\mu(S)\,\psi^{-1}\left(\frac{1}{\mu(K)\mu(S)} \int_{S} \rho_\psi\Bigl(\phi'\bigl(f(xs)-f(x)\bigr)\Bigr)\,ds\right)\\
&\leq \mu(K)\mu(S)\,\psi^{-1}\left(\frac{D-1}{\mu(K)\mu(S)} \int_{S} \rho_\phi\bigl(f(xs)-f(x)\bigr)\,ds\right)\\
&\leq \mu(K)\mu(S)\,\psi^{-1}\left(\frac{D-1}{\mu(K)\mu(S)} \rho_{\phi,S}(f)\right)<+\infty.
\end{align*}
Since $G$ is locally compact and $\mu$ is positive on open sets by \eqref{MedidaDoblante},
$$\int_S\Bigl|\phi'\bigl(f(xs)-f(x)\bigr)\Bigr|ds<+\infty$$
for almost every $x\in G$ and thus $\Delta_\phi f$ is defined almost everywhere. Furthermore, the previous estimate shows that $\Delta_\phi f\in L_{loc}^1(G)$.
\end{proof}

We want to prove the following result:

\begin{theorem}\label{TeoArmonicas}
Suppose that $\psi$ is also doubling. Then for every $[f]\in \mathcal{D}_\phi(G)$ there exists $[u]\in \overline{\mathcal{L}^\phi(G)}$ and a $\phi$-harmonic class $[h]\in \mathcal{D}_\phi(G)$ such that $[f]=[u]+[h]$.
\end{theorem}

This theorem says that every class in $\overline{H}^1\bigl(G,L^\phi(G)\bigr)$ can be represented by a $\phi$-harmonic function (unique up to constants). This also gives a one-to-one correspondence between $\phi$-harmonic classes for two different generators. To prove it, we adapt the argument used in \cite{KP2} for discrete groups. It is also suggested for the $L^p$ case in \cite{T}.

The problem of finding $[u]$ as in Theorem \ref{TeoArmonicas} is, as in other contexts, equivalent to minimizing a kind of energy operator.
Given $f\in D_\phi(G)$, we define the operator
$$\mathcal{I}^f:\overline{\mathcal{L}^\phi(G)}\to [0,+\infty),\ \mathcal{I}^f\bigl([g]\bigr)=\rho_{\phi,S}(f-g).$$
Since $\phi$ is strictly convex, $\mathcal{I}^f$ is also strictly convex. Using Proposition~\ref{PropDoblante}, it is easy to see that $\mathcal{I}^f$ is continuous.

\begin{proposition}\label{eqArm}
The class $[u]\in \overline{\mathcal{L}_\phi(G)}$ minimizes $\mathcal{I}^f$ if and only if $[h]=[f]-[u]$ is $\phi$-harmonic.
\end{proposition}

In order to prove Proposition~\ref{eqArm}, recall the definition of Gâteaux derivative of an operator and some properties involving it. 

The \textit{Gâteaux derivative} of a function $F:V\to\R$ (defined on a topological vector space $V$) at $u$ in the direction $v\in V$ is, if it exists,
$$F'(u;v)=\lim_{\lambda\to 0^+} \frac{F(u+\lambda v)-F(u)}{\lambda}.$$
We say that $F$ is \textit{Gâteaux differentiable} at $u$ if the limit exists for every $v\in V$ and the map $F_u=F'(u;\cdot)$ is in the dual space of $V$.

\begin{remark}\label{Torta}
Observe that if $F:V\to\R$ is a Gâteaux differentiable function that has a minimum  on a subspace $W$ at $u$, then $F_u|_W\equiv 0$. This is because for every $w\in W$,
$$0\leq F_u(-v)=-F_u(v)\leq 0.$$
\end{remark}

\begin{lemma}\label{IesGd}
The operator $\mathcal{I}^f$ is Gâteaux differentiable at every $[u]$ and $$\mathcal{I}^f_{[u]}\bigl([g]\bigr)=\int_S\int_G \phi'\bigl((f-u)(xs)-(f-u)(x)\bigr)\bigl(g(xs)-g(x)\bigr)\,dx\,ds.$$
\end{lemma}

\begin{proof}
Since $\phi$ is convex, for $a,b\in\R$ and $\lambda>0$ we have
\begin{align*}
\frac{\phi(a+\lambda b)-\phi(a)}{\lambda} &= \frac{\phi\bigl((1-\lambda)a +\lambda (a+b)\bigr)-\phi(a)}{\lambda}\\
&\leq \phi(a+b)-\phi(a)\leq \phi(a+b).
\end{align*}
If we put $a=(f-u)(xs)-(f-u)(x)$ and $b= g(xs)- g(x)$, we have that the function 
$$(s,x)\mapsto \frac{\phi\bigl((f-u+\lambda g)(xs)-(f-u+\lambda g)(x)\bigr)-\phi\bigl((f-u)(xs)-(f-u)(x)\bigr)}{\lambda}$$
is dominated by a function in $L^1(S\times G)$. Observe that if $(f-u)(xs)-(f-u)(x)\neq 0$, then the previous quotient goes to $\phi'\bigl((f-u)(xs)-(f-u)(x)\bigr)\bigl(g(xs)-g(x)\bigr)$ when $\lambda\to 0^+$. If $(f-u)(xs)-(f-u)(x)=0$, then the quotient is equal to
$$\frac{\phi\Bigl(\lambda\bigl(g(xs)-g(x)\bigr)\Bigr)}{\lambda},$$
which converges to $\phi'(0)=0$ because $\phi$ is an $N$-function. By the Dominated Convergence Theorem, we obtain
$$(\mathcal{I}^f)'\bigl([u];[g]\bigr)=\int_S\int_G \phi'\bigl((f-u)(xs)-(f-u)(x)\bigr)\bigl(g(xs)-g(x)\bigr)\,dx\,ds.$$

Applying H\"older's inequality \eqref{Holder}, we get
\begin{align*}
\Bigl|(\mathcal{I}^f)'\bigl([u];[g]\bigr)- &(\mathcal{I}^f)'\bigl([u];[\tilde{g}]\bigr)\Bigr|\\
&\leq \int_S\int_G \Bigl|\phi'\bigl((f-u)(xs)-(f-u)(x)\bigr)\Bigr|\bigl|(g-\tilde{g})(xs)-(g-\tilde{g})(x)\bigr|\,dx\,ds\\
&\leq 2 \bigl\|\phi'\bigl(\pi(\cdot)(f-u)-(f-u)\bigr)\bigr\|_{L^\psi(S\times G)}\|g-\tilde{g}\|_{\phi,S}
\end{align*}
By Lemma \ref{LemmaConjugada} (applied to the space $S\times G$), we have 
$$\bigl\|\phi'\bigl(\pi(\cdot)(f-u)-(f-u)\bigl)\bigr\|_{L^\psi(S\times G)}\preceq  \|f-u\|_{\phi,S}<+\infty,$$
from which we deduce that $(\mathcal{I}^f)'\bigl([u];\cdot\bigr)$ is a element of the dual space of $\overline{\mathcal{L}^\phi(G)}$, which finishes the proof.
\end{proof}

\begin{lemma}\label{LemmaSiiArm}
The class $[h]=[f]-[u]$ is $\phi$-harmonic if and only if $\mathcal{I}_{[u]}^f\equiv 0$.
\end{lemma}

\begin{proof}
($\Rightarrow$) Since $\mathcal{I}^f_{[u]}$ is continuous, it is enough to prove that $\mathcal{I}^f_{[u]}\bigl([g]\bigr)=0$ for every $g\in L^\phi(G)$.

By Young's inequality \eqref{YoungInequality} and Lemma \ref{LemmaConjugada} applied to the space $S\times G$, we have
\begin{align*}
\int_S\int_G \Bigl|\phi'\bigl(h(xs)-h(x)\bigr)\Bigr|\bigl|g(xs)\bigr|\,dx\,ds &\leq \int_S\int_G \psi\Bigl(\phi'\bigl(h(xs)-h(x)\bigr)\Bigr)\, dx\,ds + \int_S\int_G \phi\bigl(g(xs)\bigl) \,dx\,ds\\
& \leq (D-1)\rho_{\phi,S}(h)+\mu(S)\rho_\phi(g)<+\infty,
\end{align*}
where $D$ is a doubling constant for $\phi$. In the last inequality, we use that $G$ is unimodular. In the same way, we get 
$$\int_S\int_G \Bigl|\phi'\bigl(h(xs)-h(x)\bigr)\Bigr|\bigl|g(x)\bigr|\,dx\,ds<+\infty.$$
This allows to decompose $\mathcal{I}^f_{[u]}\bigl([g]\bigr)$ as follows:
\begin{align*}
\mathcal{I}^f_{[u]}\bigl([g]\bigr) &=\int_S\int_G\phi'\bigl(h(xs)-h(x)\bigr)g(xs)\,dx\,ds-\int_S\int_G\phi'\bigl(h(xs)-h(x)\bigr)g(x)\,dx\,ds\\
& =-2\int_G \Delta_\phi h(x)g(x)=0.
\end{align*}

($\Leftarrow$) For $x\in G$ and $\epsilon>0$ we define
$$\delta_{x,\epsilon}=\frac{1}{\mu\bigl(B(x,\epsilon)\bigr)}\mathbbm{1}_{B(x,\epsilon)}\in L^\phi(G).$$
As before,
$$0=\mathcal{I}^f_{[u]}(\delta_{x,\epsilon})=\frac{-2}{\mu(B(x,\epsilon))}\int_{B(x,\epsilon)} \Delta_\phi h(x)dx.$$
Applying the Differentiation Lebesgue Theorem (see \cite[Theorem 1.8]{Hei}\footnote{In \cite{Hei}, the theorem is proven for non-negative functions and doubling measure, but it can be easily generalized to our case.}) we conclude that, for almost every $x\in G$,
$$0=\lim_{\epsilon\to 0}\mathcal{I}^f_{[u]}(\delta_{x,\epsilon})=\Delta_\phi h(x);$$
thus, $h$ is $\phi$-harmonic.
\end{proof}

The last ingredient we need for proving Proposition \ref{eqArm} is the following result, which can be found in \cite[p. 24; Proposition 5.4.]{ET}.

\begin{proposition}\label{PropET2}
Let $F$ a Gâteaux differentiable function defined on a convex set $\mathcal{C}$. Then $F$ is strictly convex on $\mathcal{C}$ if and only if for every $u,v\in\mathcal{C}$ with $u\neq v$,
$$F(v)>F(u)+F'_u(v-u).$$
\end{proposition}

\begin{proof}[Proof of Proposition \ref{eqArm}]
If $\mathcal{I}^f$ has a minimum at $[u]$, then $\mathcal{I}_{[u]}^f\equiv 0$ by Remark \ref{Torta}. Using Lemma \ref{LemmaSiiArm}, we conclude that $[h]$ is $\phi$-harmonic.

Conversely, if $[h]$ is $\phi$-harmonic, then $\mathcal{I}^f_{[u]}\equiv 0$ (again by Lemma \ref{LemmaSiiArm}). Applying Proposition \ref{PropET2} to $F=\mathcal{I}^f$, we see that this operator has a minimum at $[u]$. 
\end{proof} 

In order to prove Theorem \ref{TeoArmonicas}, we use the following proposition, which is a particular case of   \cite[p. 35; Proposition 1.2]{ET}.

\begin{proposition}\label{PropET}
Let $V$ be a reflexive Banach space and $F:V\to \R$ a convex lower semicontinuous operator such that $F(u)\to +\infty$ if $\|u\|\to +\infty$.
Then $F$ has a minimum. If $F$ is in addition strictly convex, then the minimum is unique. 
\end{proposition}

\begin{proof}[Proof of Theorem \ref{TeoArmonicas}]
By Proposition \ref{eqArm}, we have to prove that $\mathcal{I}^f$ has a unique minimum. For this end we will apply Proposition \ref{PropET}.

Observe that there is a natural isometric embedding $$\mathcal{D}_\phi(G)\to L^\phi(S\times G),\ [f]\mapsto \pi(\cdot)f-f.$$
Since $\phi$ is doubling and has doubling conjugate, $L^\phi(S\times G)$ is reflexive (see \cite[p. 111, Corollary 9]{RR}). Thus $\mathcal{D}_\phi(G)$ is also reflexive because it is isometric to a closed subspace of a reflexive space. 

We already know that $\mathcal{I}^f$ is strictly convex. Furthermore, it is continuous and hence it is lower semicontinuous. Let us prove that $\mathcal{I}^f\bigl([g]\bigr)\to +\infty$ if $\|g\|_{\phi,S}\to +\infty$. 

If $\|g_n\|_{\phi,S}\to +\infty$, then, assuming that $\|f-g_n\|_{\phi,S}\geq 1$, we have
$$1=\rho_{\phi,S}\left(\frac{f-g_n}{\|f-g_n\|_{\phi,S}}\right)\leq \frac{\rho_{\phi,S}(f-g_n)}{\|f-g_n\|_{\phi,S}}= \frac{\mathcal{I}^f\bigl([g_n]\bigr)}{\|f-g_n\|_{\phi,S}},$$
and as a consequence $\mathcal{I}^f\bigl([g_n]\bigr)\geq \|f-g_n\|_{\phi,S}\to+\infty$.

Putting all together, we conclude that $\mathcal{I}^f$ has a unique minimum $[u]$, from which we obtain the desired decomposition. 
\end{proof}

\begin{remark}
Following Remark \ref{generalizacion}, we can give a definition of $\phi$-Laplacian for more general metric spaces:   
$$\Delta_{\phi,t}:D_\phi(X)\to L_{loc}^1(X),\ \Delta_{\phi,t}f(x)=\int_{X^2_t}\phi'\bigl(f(y)-f(x)\bigr)\,dy\,dx.$$
This notion is similar to the one defined in \cite{T}. All done above works in this more general context if the measure on $X$ is locally doubling. 
\end{remark}

\subsection{Examples}
{\bf (1)}
We study the case $G=\R$ with the usual addition, measure, and metric. Here $S=[-1,1]$ and $\phi$ is as in Theorem \ref{TeoArmonicas}.

On the one hand, it is easy to see that $H^1\bigl(\R,L^\phi(\R)\bigr)\neq 0$. Indeed, if $f:\R\to\R$ is a continuous increasing function such that $f(x)=0$ for every $x\leq 0$ and $f(x)=1$ for every $x\geq 1$, then it is clear that $f\in D_\phi(\R)$ because the function $x\mapsto f(x+s)-f(x)$ has image included in $[-1,1]$ and support in $[-1,2]$. It is also easy to see that $f\notin L^\phi(\R)+\R$, which implies that $f$ represents a non-zero class in $H^1\bigl(\R,L^\phi(\R)\bigr)$ via identification \eqref{Identificacion} and Theorem \ref{MainEquivalence}.

On the other hand, since $\R$ and $\Z$ are quasi-isometric, their reduced asymptotic $L^\phi$-cohomologies are isomorphic, and therefore, they have isomorphic reduced continuous $L^\phi$-cohomologies. Let us prove that $\Z$ has no $\phi$-harmonic classes, which implies $\overline{H}^1\bigl(\Z,\ell^\phi(\Z)\bigr)=\overline{H}^1\bigl(\R,L^\phi(\R)\bigr)=0$. In particular, $\R$ has no non-trivial $\phi$-harmonic classes. 

The argument below can also be found in \cite{Pa}. 

Consider in $\Z$ the generator $S=\{-1,0,1\}$. If $f\in D_\phi(\Z)$ is $\phi$-harmonic, then for every $n\in\Z$,
$$0=\Delta_\phi f(n)=\phi'\bigl(f(n+1)-f(n)\bigr)+\phi'\bigl(f(n-1)-f(n)\bigr).$$
Since $\phi'$ is odd and increasing, from the previous equality we have that $n\mapsto f(n+1)-f(n)$ is constant. Which implies that $f$ is constant because $f$ is a $\phi$-Dirichlet function. We conclude that the only $\phi$-harmonic class on $\Z$ is the trivial one.

\medskip

{\bf (2)}
Let us say something about the $L^\phi$-cohomology of the real hyperbolic space $\H^n$ for some fixed doubling Young function $\phi$. It can be seen as the Heintze group $\H^n=\R^{n-1}\rtimes_{Id}\R$ (see \cite{Heintze}).

We first observe that if $\Gamma\leq \mathrm{Isom}(\H^n)$ is a discrete group such that $M=\H^n/\Gamma$ is a closed hyperbolic manifold, then $\Gamma$ acts freely, properly discontinuously, and cocompactly on $\H^n$. Moreover, a simplicial structure can be defined by lifting a triangulation of $M$ to $\H^n$. According to this structure, $\Gamma$ acts also by simplicial automorphisms; hence, Proposition \ref{equivSG} implies that $H^k\bigl(\Gamma,\ell^\phi(\Gamma)\bigr)$ and $\ell^\phi H^k(\H^n)$ are isomorphic for every $k\in\N$. The same is true for the reduced cohomology.

If we equip $\Gamma$ with the word metric and the counting measure, it satisfies the hypothesis of Theorem \ref{MainEquivalence}.  The groups $\Gamma$ and $\H^n$ are quasi-isometric and hence by Corollary \ref{MainInvariance} their (reduced) continuous $L^\phi$-cohomologies coincide (and they coincide with their asymptotic $L^\phi$-cohomologies). Therefore,
$$H^k(\H^n,L^\phi\bigl(\H^n)\bigr)\cong\ell^\phi H^k(\H^n)\text{ and }\overline{H}^k\bigl(\H^n,L^\phi(\H^n)\bigr)\cong\ell^\phi \overline{H}^k(\H^n),\quad k\in\N.$$

In the case $k=1$, Theorem 1.2 in \cite{C} implies that $\ell^\phi H^1(\H^n)=\ell^\phi \overline{H}^1(\H^n)$ and they coincide with the Besov space $\mathcal{B}_\phi(\mathbb{S}^{n-1})/\R$, where
$$\mathcal{B}_\phi(\mathbb{S}^{n-1})=\{u:\mathbb{S}^{n-1}\to\R : \|u\|_{\mathcal{B}_\phi}<+\infty\}$$
and $\|\ \|_{\mathcal{B}_\phi}$ is the Luxembourg semi-norm associated to 
$$\rho_{\mathcal{B}_\phi}(u)=\int_{\mathbb{S}^{n-1}\times\mathbb{S}^{n-1}} \frac{\phi\bigl(u(x)-u(y)\bigr)}{|x-y|^{2n-2}}\,d\mathcal{H}(x)\,d\mathcal{H}(y).$$
Here $\mathcal{H}$ is the $(n-1)$-dimensional Hausdorff measure on the sphere, and, as before, $\R$ denotes the space of constant functions. 

If $\phi(t)=|t|^p$, then it is easy to see that the Lipschitz functions on an Ahlfors-regular metric space $Z$ are in $\mathcal{B}_\phi(Z)$ if $p$ is greater than the Hausdorff dimension of $Z$, which implies $\mathcal{B}_\phi(Z)/\R\neq 0$. Let us repeat the proof in our more general case in order to obtain some condition on $\phi$ for the non-vanishing of $\mathcal{B}_\phi(\mathbb{S}^{n-1})/\R$. 

Let $u:\mathbb{S}^{n-1}\to\R$ be a $L$-Lipschitz function. The sphere is $(n-1)$-Ahlfors regular, that is, there exists $C\geq 1$ such that for every $x\in\mathbb{S}^{n-1}$ and $r\in (0,2\pi)$,
\begin{equation}\label{AR}
C^{-1}r^{n-1}\leq\mathcal{H}\bigl(B(x,r)\bigr)\leq C r^{n-1}.    
\end{equation}
Here we assume that $\mathbb{S}^{n-1}$ has diameter $2\pi$. Define the $m$-annulus around a point $x\in\mathbb{S}^{n-1}$ as the subset $A_m(x)=B\bigl(x,\frac{2\pi}{m}\bigr)\setminus B\bigl(x,\frac{2\pi}{m+1}\bigr)$. Then, by \eqref{AR},
\begin{align*}
\rho_{\mathcal{B}_\phi}(u) &= \int_{\mathbb{S}^{n-1}}\sum_{m\geq 1} \left( \int_{A_m(x)} \frac{\phi\bigl(u(x)-u(y)\bigr)}{|x-y|^{2n}} \,d\mathcal{H}(y)\right)d\mathcal{H}(x)\\
&\leq \int_{\mathbb{S}^{n-1}}\sum_{m\geq 1} \mathcal{H}\bigl(A_m(x)\bigr) \phi(2\pi L/m)(m+1)^{2n-2} \,d\mathcal{H}(x)\\
& \leq \mathcal{H}(\mathbb{S}^{n-1}) \sum_{m\geq 1}\phi(2\pi L/m) \left(\frac{2\pi}{m}\right)^{n-1}(m+1)^{2n-2}.
\end{align*}
Thus, a sufficient condition to $u$ be in $\mathcal{B}_\phi(\mathbb{S}^{n-1})$ is 
\begin{equation}\label{A}
\sum_{m\geq 1} \phi(1/m)m^{n-1}<+\infty.    
\end{equation}

For any fixed point $x_0\in\mathbb{S}^{n-1}$, the map $u(x)=|x-x_0|$ is Lipschitz and non-constant. If $\phi$ satisfies \eqref{A}, then $\ell^\phi H^1(\H^n)=\ell^\phi \overline{H}^1(\H^n)\neq 0$. We conclude that
$$H^1\bigl(\H^n,L^\phi(\H^n)\bigr)=\overline{H}^1\bigl(\H^n,L^\phi(\H^n)\bigr)\neq 0.$$


A condition similar to \eqref{A} is given in \cite{KG} as a sufficient condition for the non-vanishing of the de Rham Orlicz cohomology of $\H^2$ in degree $1$.

Observe that the Haar measure on $\mathbb{H}^n$ is the Riemannian volume, hence it is locally doubling. Therefore, if $\phi$ and $S$ are as in Subsection \ref{SectionHarmonic}, then condition \eqref{A} guarantees the existence of non-constant $\phi$-harmonic functions.

An explicit computation of the simplicial Orlicz cohomology in degree $1$ of a wide family of Heintze groups for certain doubling Young functions can be found in \cite{C}.  

\section{Some observations on the non-doubling case}

In this section, we study an example that illustrates some differences between the doubling and non-doubling case.

Consider the free group $F_2$ generated by two generators $a$ and $b$. We equip $F_2$ with the counting measure and the word metric associated to the symmetric generator $S=\{a,a^{-1},b,b^{-1}\}$.

Let us focus on the asymptotic Orlicz cohomology of $F_2$ associated to a Young function $\phi$. Observe that for every $x,y\in F_2$ there exists $n\in\N$ and  $x_0,x_1,\ldots,x_n\in F_2$ (all of them different) such that $x_0=x$, $x_n=y$ and $|x_{i-1}-x_{i}|=1$; moreover, these points are unique. For $\omega\in Z_\phi^1(F_2)$, we have
$$\omega(x,y)=\sum_{i=1}^n \omega(x_{i-1},x_i).$$
We can conclude that every element in $Z_\phi^1(F_2)$ is determined by its values at the set $(F_2^2)_1$ of all the pairs of elements at distance $1$, which also implies that $\|\ \|_{\phi,1}$ is a norm in $Z_\phi^1(F_2)$.

Let $X$ be the Cayley graph of $F_2$ for the generator $S$, which is geometrically a tree. It is clear that the map
$$\Theta:\Bigl(Z_\phi^1(F_2),\|\ \|_{\phi,1}\Bigr)\to \Bigl(\ell^\phi\bigl(X^{(1)}\bigr),\|\ \|_\phi\Bigr),\ \Theta(\omega)\bigl([x,y]\bigr)=\omega(x,y),$$
is an isomorphism that preserves $d\ell^\phi(F_2)$. In particular, $\bigl(Z_\phi^1(F_2),\|\ \|_{\phi,1}\bigr)$ is a Banach space and has the topology given by the whole family of semi-norms $\|\ \|_{\phi,t}$. This shows that the (reduced) asymptotic $L^\phi$-cohomology of $F_2$ coincides with the (reduced) simplicial $\ell^\phi$-cohomology of $X$ even if $\phi$ is not doubling. 

Consider a function $\phi$ such that $\phi(t)=e^{-\frac{1}{t^2}}$ for $|t|$ small enough. It is easy to see that this formula defines a convex function on $\bigl(-\sqrt{2/3},\sqrt{2/3}\bigr)$. Since we are in the discrete case, the behaviour of the function for large $t$ is not important. However, $\phi$ can be extended to a non-doubling Young function on $\R$, for example by putting $\phi(t)=\alpha+\beta e^{|t|}$ when $|t|>\sqrt{2/3}$ for suitable $\alpha,\beta\in\R$.

First observe that $L^\phi H_{AS}^1(F_2)\neq 0$. For that we decompose $F_2$ into two disjoint subsets $A$ and $B$, where $A$ is the set of elements $x\in F_2$ that can be written as $x=a s_1 \cdots s_k$ with $s_1,\ldots,s_k\in S$ and $s_1^{-1}\neq a$. Take $\omega\in Z_\phi^1(F_2)$ defined in $(F_2^2)_1$ by $\omega(1,a)=\epsilon$ (and then $\omega(a,1)=-\epsilon$) and $\omega(x,y)=0$ if $\{x,y\}\neq \{1,a\}$, where $\epsilon>0$. If $f:F_2\to\R$ satisfies $df=\omega$, then $f$ must be constant on $A$ and $B$ but taking a different value on each subset, so it cannot be in $\ell^\phi(F_2)$. This implies that $\omega$ represents a non-zero class in cohomology.

Now check that if $\epsilon<\sqrt{2/3}$, then $\omega$ can be approximated by a sequence $\{\omega_n\}\subset Z_\phi^1(F_2)$ such that for every $n\in \N$ there exists $f_n\in\ell^\phi(F_2)$ with $df_n=\omega_n$. We again define $\omega_n$ in $(F_2^2)_1$ such that
\begin{itemize}
    \item $\omega_n(1,a)=\epsilon$
    \item $\omega_n(x,y)=\epsilon/n$ if $x,y\in A$ and $|x-a|=|y-a|-1\leq n-1$ 
    
    \item $\omega_n(x,y)=0$ if $x,y\in B$ or $|x-a|>n$ or $|y-a|>n$. 
\end{itemize}

It is clear that $\omega_n=df_n$ for $f_n$ with finite support. Moreover
$$\rho_{\phi,1}\left(\frac{\omega_n-\omega}{\alpha}\right)=2\cdot 3^n e^{-\left(\frac{\alpha n}{\epsilon}\right)^2},$$
which is equal to $1$ if
$$\alpha=\epsilon\sqrt{\frac{\log 2}{n^2}+\frac{\log 3}{n}}.$$
This shows that $\|\omega_n-\omega\|_\phi\to 0$ when $n\to +\infty$.

It is known (see~\cite[Proposition~2]{Ko13}) that in~the~doubling case, the~continuous $L^\phi$-cohomology in degree $1$ of~a~noncompact second countable locally compact group coincides with its reduced cohomology if and only if the~group is non-amenable. By Theorem \ref{MainEquivalence} the same holds for the asymptotic Orlicz cohomology. However, the above observation shows that the asymptotic $L^\phi$-cohomology in degree $1$ of the non-amenable group $F_2$ can be non-reduced (that is $L^\phi H_{AS}^1(F_2)\neq L^\phi \overline{H}_{AS}^1(F_2)$) if $\phi$ is non-doubling. 


Theorem 1.2 in \cite{C} implies that, if $\phi$ is doubling and $X$ is a Gromov-hyperbolic simplicial complex with bounded geometry such that its boundary $\partial X$ admits an Ahlfors-regular visual metric, then $\ell^\phi H^1(X)=\ell^\phi \overline{H}^1(X)$.  In our case, it is easy to see that the Cayley graph $X$ is Gromov-hyperbolic and its boundary has an Ahlfors-regular visual metric of dimension $\log 3$. Then, combining this result with Theorem \ref{MainEquivalence} and Proposition \ref{equivSG}, we obtain $L^\phi H_{AS}(X)=L^\phi \overline{H}_{AS}(X)$ if $\phi$ is doubling. Observe that the above computation shows that this is not true in the non-doubling case. 

In fact, we can see directly that if $\phi$ is as above, then $\ell^\phi H^1(X)\neq\ell^\phi \overline{H}^1(X)$, which shows that the doubling condition is necessary for the claim of  \cite[Theorem 1.2]{C}.

\bibliographystyle{plainurl}
\bibliography{references.bib}

\begin{thebibliography}{10}

\bibitem{A}
M.~Atiyah.
\newblock Elliptic operators, discrete groups and von neumann algebras.
\newblock {\em Ast\'erisque}, 32.33:43--72, 1976.
\newblock URL:
  \url{http://www.numdam.org/item/AST\_1976\_\_32-33\_\_43\_0.pdf}.

\bibitem{B}
P.~Blanc.
\newblock Sur la cohomologie continue des groupes localement compacts.
\newblock {\em Ann. Sci. \'Ec.Norm. Sup.}, 12(2):137--168, 1979.
\newblock URL:
  \url{http://www.numdam.org/article/ASENS\_1979\_4\_12\_2\_137\_0.pdf}.

\bibitem{BMV05}
M.~Bourdon, F.~Martin, and A.~Valette.
\newblock Vanishing and non-vanishing for the first ${L}_p$-cohomology of
  groups.
\newblock {\em Comment. Math. Helv.}, 80(2):377--389, 2005.
\newblock \href {https://doi.org/10.4171/CMH/18} {\path{doi:10.4171/CMH/18}}.

\bibitem{BP}
M.~Bourdon and H.~Pajot.
\newblock Cohomologie $l_p$ et espaces de {B}esov.
\newblock {\em J. Reine Angew Math.}, 558:85--108, 2003.
\newblock \href {https://doi.org/10.1515/crll.2003.043}
  {\path{doi:10.1515/crll.2003.043}}.

\bibitem{BR}
M.~Bourdon and B.~R\'emy.
\newblock Quasi-isometric invariance of continuous group ${L}^p$-cohomology,
  and first applications to vanishings.
\newblock {\em Ann. Henri Lebesgue}, 3:1291--1326, 2020.
\newblock \href {https://doi.org/10.5802/ahl.61} {\path{doi:10.5802/ahl.61}}.

\bibitem{BR2}
M.~Bourdon and B.~R\'emy.
\newblock Non-vanishing for group ${L}^p$-cohomology of solvable and semisimple
  lie groups.
\newblock {\em J. Éc. polytech., Math.}, 10:771--814, 2023.
\newblock \href {https://doi.org/10.5802/jep.232} {\path{doi:10.5802/jep.232}}.

\bibitem{BH}
M.~Bridson and A.~Heafliger.
\newblock {\em Metric Spaces of Non-Positive Curvature.}
\newblock Grundlehren der Mathematischen Wissenschaften. Springer-Verlag
  Berlin, Heidelberg, 1999.

\bibitem{C}
M.~Carrasco.
\newblock Orlicz spaces and the large scale geometry of heintze groups.
\newblock {\em Math. Ann.}, 368:433--481, 2016.
\newblock \href {https://doi.org/10.1007/s00208-016-1430-1}
  {\path{doi:10.1007/s00208-016-1430-1}}.

\bibitem{CH}
Y.~Cornulier and P.~De la~Harpe.
\newblock {\em Metric Geometry of Locally Gompact Groups.}
\newblock EMS Tracts in Mathematics, vol 25. European Mathematical Society,
  2016.

\bibitem{CT}
Y.~Cornulier and R.~Tessera.
\newblock Contracting automorphisms and ${L}^p$-cohomology in degree one.
\newblock {\em Ark. Mat.}, 49(2):295--324, 2011.
\newblock \href {https://doi.org/10.1007/s11512-010-0127-z}
  {\path{doi:10.1007/s11512-010-0127-z}}.

\bibitem{ET}
I.~Ekeland and R.~T\'emam.
\newblock {\em Convex Analysis and Variational Problems.}
\newblock Society for Industrial and Applied Mathematics, Philadelphia, 1999.

\bibitem{E}
G.~Elek.
\newblock The $\ell_p$-cohomology and the conformal dimension of hyperbolic
  cones.
\newblock {\em Geom. Ded.}, 8:263--279, 1997.
\newblock \href {https://doi.org/10.1023/A:1004920322337}
  {\path{doi:10.1023/A:1004920322337}}.

\bibitem{E2}
G.~Elek.
\newblock Coarse cohomology and $\ell^p$-cohomology.
\newblock {\em K-Theory}, 13(1):1--22, 1998.
\newblock \href {https://doi.org/10.1023/A:1007735219555}
  {\path{doi:10.1023/A:1007735219555}}.

\bibitem{G}
L.~Genton.
\newblock {\em Scaled Alexander-Spanier Cohomology and $L_{qp}$-Cohomology for
  Metric Spaces.}
\newblock Thesis Nº 6330. EPFL, Lausanne, 2014.
\newblock URL: \url{https://infoscience.epfl.ch/record/201809}.

\bibitem{GH}
E.~Ghys and P.~de~la Harpe.
\newblock {\em Sur les Groupes Hyperboliques d'après Mikhael Gromov.}
\newblock Progress in Mathematics, vol 83. Birkhäuser Boston, 1990.

\bibitem{Gol}
H.~Goldman.
\newblock {\em Uniform Fr\'echet Algebras.}
\newblock North-Holland Mathematics Studies, vol 162. North-Holland Publishing
  Co., Amsterdam, 1990.

\bibitem{KG}
V.~Gol$'$dshtein and Ya. Kopylov.
\newblock Some calculations of {O}rlicz cohomology and
  {P}oincar{\'e}--{S}obolev--{O}rlicz inequalities.
\newblock {\em Sib. \`Elektron. Mat. Izv.}, 16:1079--1090, 2019.
\newblock \href {https://doi.org/10.33048/semi.2019.16.075}
  {\path{doi:10.33048/semi.2019.16.075}}.

\bibitem{GKSh82_1}
V.~Gol$'$dshtein, V.~Kuz$'$minov, and I.~Shvedov.
\newblock Differential forms on {L}ipschitz manifolds.
\newblock {\em Sib. Math. J.}, 23(2):151--161, 1982.
\newblock \href {https://doi.org/10.1007/BF00971687}
  {\path{doi:10.1007/BF00971687}}.

\bibitem{GKSh82_2}
V.~Gol$'$dshtein, V.~Kuz$'$minov, and I.~Shvedov.
\newblock Integration of differential forms of the classes ${W}^*_{p,q}$.
\newblock {\em Sib. Math. J.}, 23(5):640--653, 1982.
\newblock \href {https://doi.org/10.1007/BF00971282}
  {\path{doi:10.1007/BF00971282}}.

\bibitem{G93}
M.~Gromov.
\newblock {\em Asymtotic Invariants for Infinite Groups.}
\newblock G. A. Niblo and M. A. Roller, eds., London Math. Soc. Lect. Note Ser.
  182, 1993.

\bibitem{Gui}
A.~Guichardet.
\newblock {\em Cohomologie des Groupes Topologiques et des Algèbres de Lie.}
\newblock Cedic/F. Nathan, Paris, 1980.

\bibitem{Hei}
J.~Heinonen.
\newblock {\em Lectures on Analysis on Metric Spaces.}
\newblock Universitext. Sprinberg-Verlag New York, 2001.

\bibitem{Heintze}
E.~Heintze.
\newblock On homogeneous manifolds of negative curvature.
\newblock {\em Math. Ann.}, 211:23--34, 1974.
\newblock \href {https://doi.org/10.1007/BF01344139}
  {\path{doi:10.1007/BF01344139}}.

\bibitem{Ko13}
Ya. Kopylov.
\newblock Amenability of closed subgroups and {O}rlicz spaces.
\newblock {\em Sib. \`Elektron. Mat. Izv.}, 10:583--590, 2013.
\newblock URL:
  \url{http://www.mathnet.ru/links/25e46dcf491cee8c5f9a36db9d053bbd/semr452.pdf}.

\bibitem{KP2}
Ya. Kopylov and R.~Panenko.
\newblock $\phi$-harmonic functions on discrete groups and the first
  $\ell^\phi$-cohomology.
\newblock {\em Sib. Math. J.}, 55:904--914, 2014.
\newblock \href {https://doi.org/10.1134/S0037446614050097}
  {\path{doi:10.1134/S0037446614050097}}.

\bibitem{MT}
J.~Mackay and J.~Tyson.
\newblock {\em Conformal Dimension: Theory and Application.}
\newblock University Lecture Series; Vol. 54. American Mathematical Society,
  2010.

\bibitem{MV}
F.~Martin and A.~Valette.
\newblock On the first ${L}^p$-cohomology of discrete groups.
\newblock {\em Groups Geom. Dyn.}, 1(1):81--100, 2007.
\newblock \href {https://doi.org/10.4171/ggd/5} {\path{doi:10.4171/ggd/5}}.

\bibitem{Pa}
R.~Panenko.
\newblock {\em Orlicz Spaces on Groups, Manifolds, and Graphs.}
\newblock Thesis, Novosibirsk. 2018.

\bibitem{Pa88}
P.~Pansu.
\newblock Cohomologie $l^p$ des vari\'et\'es \`a courbure n\'egative, cas du
  degr\'e $1$.
\newblock {\em Rend. Sem. Mat. Univ. Pol. Torino, Special Issue}, 1989.

\bibitem{Pa95}
P.~Pansu.
\newblock Cohomologie ${L}^p$: invariance sous quasiisométries.
\newblock {\em Preprint}, 1995.
\newblock URL:
  \url{https://www.imo.universite-paris-saclay.fr/~pansu/qi04.pdf}.

\bibitem{Puls}
M.~Puls.
\newblock The first ${L}^p$-cohomology of some finitely generated groups and
  $p$-harmonic functions.
\newblock {\em J. Funct. Anal.}, 237:391--401, 2006.
\newblock \href {https://doi.org/10.1016/j.jfa.2006.04.011}
  {\path{doi:10.1016/j.jfa.2006.04.011}}.

\bibitem{RR}
M.~Rao and Z.~Ren.
\newblock {\em Theory of Orlicz Spaces.}
\newblock Monographs and Textbooks in Pure and Applied Mathematics, vol 146.
  Marcel Dekker, Inc., 1991.

\bibitem{SS}
R.~Sauer and M.~Schrödl.
\newblock Vanishing of $\ell ^2$-betti numbers of locally compact groups as an
  invariant of coarse equivalence.
\newblock {\em Fund. Math.}, 243:301--311, 2018.
\newblock \href {https://doi.org/10.4064/fm512-1-2018}
  {\path{doi:10.4064/fm512-1-2018}}.

\bibitem{Seq}
E.~Sequeira.
\newblock {\em Relative $L^p$ and Orlicz Cohomology and Applications to Heintze
  Groups.}
\newblock Thesis. Universidad de la República - Université de Lille, 2020.
\newblock URL: \url{https://theses.fr/2020LILUI053}.

\bibitem{T}
R.~Tessera.
\newblock Vanishing of the first reduced cohomology with values in an
  ${L}^p$-representation.
\newblock {\em Ann. Inst. Fourier}, 59:851--876, 2009.
\newblock \href {https://doi.org/10.5802/aif.2449}
  {\path{doi:10.5802/aif.2449}}.

\bibitem{Z}
R.~Zimmer.
\newblock {\em Ergodic Theory and Semisimple Groups}.
\newblock Monographs in Mathematics, vol. 81. Birkhäuser, 1984.

\end{thebibliography}

\end{document}